\documentclass[a4paper,12pt]{article}
\usepackage{amscd,amssymb,amsmath,amsthm}
\usepackage{hyperref}
\usepackage{graphicx,caption,dsfont}
\usepackage{mathtools}

\DeclarePairedDelimiter\floor{\lfloor}{\rfloor}
\usepackage{subcaption}
\usepackage{amsfonts,enumerate}
\usepackage{fullpage}
\usepackage[numbers,sort & compress]{natbib}
\usepackage{soul}
\usepackage{booktabs}
\numberwithin{equation}{section}

\newtheorem{thm}{Theorem}[section]

\newtheorem{defn}[thm]{Definition}

\newtheorem{rem}{Remark}

\newtheorem{nt}{Note}

\newtheoremstyle{case}{}{}{}{}{}{:}{ }{}
\theoremstyle{case}

\allowdisplaybreaks

\makeatletter
\def\and{%
  \end{tabular}%
  \begin{tabular}[t]{c}}%
\def\@fnsymbol#1{\ensuremath{\ifcase#1\or 1\or 2\or 3\or
   4\or 5\or 6\or 7\or 8\or 9\else\@ctrerr\fi}}
\makeatother
\vspace{-8ex}
\date{}
\begin{document}
\baselineskip=0.70cm
{\title{\textbf{Exact solutions of generalized non-linear time-fractional reaction-diffusion equations with time delay }}
	\author{P. Prakash$^1$, Sangita Choudhary$^2$, and Varsha Daftardar-Gejji$^{2*}$\\
$^1$ Department of Mathematics,\\
			Amrita Vishwa Vidyapeetham, Coimbatore-641112, India.\\
		$^2$Department of Mathematics,\\
			Savitribai Phule Pune University, Pune-411007, India.\\
			\textit{$^1$vishnuindia89@gmail.com,\,$^2$schoudhary1695@gmail.com,}\\ 	\textit{$^{2*}$vsgejji@gmail.com,
		$^{2*}$vsgejji@math.unipune.ac.in.}
			}
		\date{}
		\maketitle
\maketitle
\begin{abstract}
In this paper, we propose the invariant subspace approach to find exact solutions of time-fractional partial differential equations (PDEs) with time delay. An algorithmic approach of finding invariant subspaces for the generalized non-linear time-fractional reaction-diffusion equations with time delay is presented. We show that the fractional reaction-diffusion equations with time delay admit several invariant subspaces which further yields several distinct analytical solutions. We also demonstrate how to derive exact solutions for time-fractional PDEs with multiple time delays. Finally, we extend invariant subspace method to more generalized time-fractional PDEs with non-linear term involving time delay.
\end{abstract}
\section{Introduction}
Many natural phenomena depend not only on the current state of the system at a particular time but also on the previous states. This memory effect can be successfully modeled by using the theory of delay differential equations (DDEs) \cite{in13,in14,ml}. Over the past few decades, the study of DDEs has helped to investigate many complex and natural non-linear phenomena in climate modeling, bioengineering \cite{bgbm}, control theory \cite{con}, agriculture \cite{agri}, traffic models \cite{tri1}, epidemiology and population dynamics \cite{in13,in14}, chemical kinetics \cite{che1}, chaos \cite{chaos}  and other areas of science and engineering \cite{hs,in13,in14,new1,ml,new2,new3}. Many physical models, especially non-linear ones, are methodically and effectively analyzed with the help of fractional calculus and in particular, with fractional delay differential equations (FDDEs) \cite{pi,kai,rh,moo,bmbg15, bg10, bgbm12}. The subject of fractional delay PDEs is rather recent and has proven to be a powerful tool in describing various natural and scientific phenomena in Science and Engineering.

In the recent years, many researchers have made attempts to find and study solutions of non-linear fractional or integer-order PDEs without delay using analytical and numerical methods such as Lie symmetry analysis method \cite{pra1,pra2,pra5}, Adomian decomposition method \cite{Ge,mo}, Hirota bilinear method \cite{dif2} and so on. In general, it is very difficult to derive exact solutions of non-linear fractional delay PDEs. Analytical solutions for most of the classical PDEs with delay are still not available. Polyanin and Zhurov \cite{rus} have investigated and found exact solutions of delay reaction-diffusion equation. Non-linear time-fractional delay reaction-diffusion equations are more complex than their integer-order counterparts, as fractional derivatives are more involved than the classical derivatives. Thus finding exact solutions is tedious and challenging task in case of non-linear fractional delay PDEs.

The present paper deals with the investigation of analytical solutions for the generalized non-linear time-fractional reaction-diffusion (RD) equation with time delay
\begin{align}
\begin{aligned}\label{crd}
&\dfrac{\partial^{\alpha}u}{\partial t^{\alpha}}=\left[D(u)u_{x}\right]_{x}+R(u,\bar{u}),\ t>0, \ \alpha\in(0,1],\\
&u(x,t)=\Phi(x,t),\ t\in [-\tau,0],
\end{aligned}
\end{align}
and non-linear time-fractional heat equation with source term (or RD equation) involving time delay
\begin{align}
\begin{aligned}\label{qrd}
&\dfrac{\partial^{\alpha}u}{\partial t^{\alpha}}=\left[D(u)u_{x}\right]_{x}+R(u,\bar{u}), t>0,\  \alpha\in(0,1],\\
&u(x,t)=\Psi(x,t),\ t\in [-\tau,0].
\end{aligned}
\end{align}
Here $u=u(x,t),\ \bar{u}=u(x,t-\tau),\ \tau>0$ and $x\in \mathbb{R}$. The term $D(u)$ is transfer/ diffusion coefficient that depends on $u$, and $R(u,\bar{u})$ denotes the rate of reactions known as kinetic function which depends on $u$ and $\bar{u}$ involving time delay. These equations are widely used to describe plenty of natural phenomena in the areas of Science and Engineering \cite{in2, in3, in4, in5}. Note that when $\alpha=1$ and $\tau=0$, the equation \eqref{qrd} can be referred as quasilinear heat equation or reaction-diffusion equation \cite{gs}. The analytical solutions of integer-order PDE without delay corresponding to \eqref{qrd}, was discussed by Galaktionov and Svirshchevskii using the invariant subspace method \cite{gs}.

To best of our knowledge, no one has investigated the exact solutions of time-fractional delay reaction-diffusion equations in the literature so far. Here, we derive exact solutions for time-fractional reaction-diffusion equations with time delay using the invariant subspace method (ISM). ISM was introduced by Galaktionov and Svirshchevskii \cite{gs}, and further developed by many researchers \cite{ps1,ps2,sc,pra4,pra3,hes,s1,pa1,pa2,ra,rk,wx1,wx,wy,vis} for integer and fractional order scalar and coupled PDEs. The main objective of the present paper is to demonstrate that the generalized non-linear time-fractional delay reaction-diffusion equations admit several invariant subspaces which further yields several distinct exact solutions. We also present the exact solution for non-linear reaction-diffusion equations with multiple time delays.

The research paper is organized as follows. In section 2, we provide some basic concepts and results that are used throughout this paper. Further we generalize the theoretical framework of ISM for solving non-linear time-fractional PDEs involving a linear term with time delay. Section 3 presents an algorithmic approach for finding invariant subspaces. Further we construct invariant subspaces of dimension n=1, 2, 3, 4, 5, and corresponding linear as well as non-linear operators for both RD equations \eqref{crd}-\eqref{qrd} under study. In section 4, we illustrate the applicability and effectiveness of the extended ISM by finding exact exponential, polynomial and trigonometric solutions of the above-mentioned time-fractional delay PDEs. In section 5, we discuss the extension of the invariant subspace method to time-fractional PDEs with multiple time delays, and find its exact solutions. Further, we employ the ISM to more generalized non-linear PDEs with time delay. Finally in section 6, conclusions are summarized.
\section{Preliminaries}
 In this section, we provide some relevant basic concepts and definitions of the fractional calculus. Further we present brief details of the ISM for time-fractional PDEs involving time delay.
 \subsection{Prerequisites of fractional calculus}
\begin{defn}[\cite{pi,kai}]
Let $\varphi(t)\in C^{n}[a,b]$ and $\alpha>0$.
Then the Caputo fractional derivative of order $\alpha>0$ is defined by
\begin{eqnarray}\label{zz1}
\nonumber \dfrac{d^{\alpha}\varphi(t)}{dt^{\alpha}}
=\left\{
\begin{array}{ll}  \dfrac{1}{\Gamma(n-\alpha)}\displaystyle\int\limits^{t}_{0}\dfrac{\varphi^{(n)}(s)}{(t-s)^{\alpha-n+1}}ds, & n-1<\alpha< n,\\
 \varphi^{(n)}(t), & \alpha=n,\ n\in\mathds{N},
\end{array}
\right.
\end{eqnarray}
where $C^{n}[a,b]$ denotes the set of all continuously $n$-times differentiable functions.
\end{defn}
\begin{defn} \cite{am}
Mittag-Leffler function with three parameters, also known as Prabhakar function, is defined as
\begin{equation}
\mathbf{E}_{\alpha,\beta}^{\gamma}(z)=\sum\limits_{k=0}^{\infty}\dfrac{(\gamma)_{k}z^{k}}{\Gamma(\alpha k+\beta)k!},\ \alpha, \beta, \gamma\in\mathds{C},\ \mathcal{R}e(\alpha)>0, \ \mathcal{R}e(\beta)>0,
\end{equation}
where $(\gamma)_{k}=\dfrac{\Gamma(\gamma+k)}{\Gamma(\gamma)}$ and  $(\gamma)_{0}=1$, $\mathcal{R}e(\gamma)>0$.
\end{defn}
\begin{rem}
The functions $E_{\alpha,1}^{1}$ and $E_{\alpha,\beta}^1$ are called as one-parameter and two-parameters Mittag-Leffler functions respectively.
\end{rem}
\begin{nt}[\cite{pi,kai}]
The Laplace transform of Caputo fractional derivative of order $\alpha\in(n-1, n], n\in\mathds{N},$ is
\begin{equation*}
L\left\{\dfrac{d^{\alpha}\varphi(t)}{dt^{\alpha}}\right\}=s^{\alpha}\hat{\varphi}(s)-\sum\limits^{n-1}_{k=0}s^{\alpha-k-1}\varphi^{(k)}(0), \  \mathcal{R}e(s)>0.
\end{equation*}
\end{nt}
\begin{nt}\cite{am}
  The Laplace transformation of the generalized Mittag-Leffler function $t^{\beta-1}\mathbf{E}_{\alpha,\beta}^{\gamma}(\pm at^{\alpha})$ is given by
\begin{equation*}
L\left\{t^{\beta-1}\mathbf{E}_{\alpha,\beta}^{\gamma}(\pm at^{\alpha})\right\}=\dfrac{s^{\alpha\gamma-\beta}}{(s^{\alpha}\mp a)^{\gamma}},    \mathcal{R}e(s)>|a|^{\frac{1}{\alpha}}.
\end{equation*}
\end{nt}
\begin{nt}[\cite{lap,joi}] Delayed unit step function or Heaviside function is defined as
\begin{equation}\label{un}
H(t-a)=\left\{
         \begin{array}{ll}
           1, & t\geq a; \\
           0, & t<a.
         \end{array}
       \right.
\end{equation}
Laplace transformation of unit step function \eqref{un} is given by
\begin{equation*}
L\left\{H(t-a)\right\}=\dfrac{e^{-a s}}{s},\ \mathcal{R}e(s)>0.
\end{equation*}
If $L\left\{\varphi(t)\right\}=\hat{\varphi}(s)$ for $\mathcal{R}e(s)>0$, then
\begin{equation*}
L\left\{H(t-a)\varphi(t-a)\right\}=e^{-as}\hat{\varphi}(s),\ a\geq0.
\end{equation*}
By taking the inverse Laplace transform of both sides, we get
\begin{equation*}
H(t-a)\varphi(t-a)=L^{-1}\left\{e^{-as}\hat{\varphi}(s)\right\}.
\end{equation*}
\end{nt}
\begin{thm}\label{conv}
If $L\left\{\varphi_{1}(t)\right\}=\displaystyle\hat{\varphi}_{1}(s)$ and $L\left\{\varphi_{2}(t)\right\}=\hat{\varphi}_{2}(s)$, then
$$L^{-1}\left\{\hat{\varphi}_{1}(s)\hat{\varphi}_{2}(s)\right\}=\varphi_{1}(t)\star \varphi_{2}(t),$$
where \textquoteleft\ $\star$\textquoteright\ denotes the convolution of $\varphi_1(t)$ and $\varphi_2(t)$, and is defined by the integral
\begin{equation*}\label{l4}
\varphi_{1}(t)\star \varphi_{2}(t)=\int\limits^{t}_{0}\varphi_{1}(t-\xi)\varphi_{2}(\xi)d\xi=\int\limits^{t}_{0}\varphi_{2}(t-\xi)\varphi_{1}(\xi)d\xi.
\end{equation*}
\end{thm}

\subsection{Invariant subspace method for non-linear time-fractional PDEs involving a linear term with time delay}
Consider the non-linear time-fractional PDE involving a linear term with time delay
\begin{align}\label{gel}
\begin{aligned}
&\dfrac{\partial^{\alpha}u}{\partial t^{\alpha}}=\mathcal{H}[u,\bar{u}]\equiv\mathcal{N}[u]+\delta\bar{u},\ \alpha>0,\ t> 0,\ \delta, x\in \mathbb{R},\\
&u(x,t)=\Phi(x,t),\ t\in [-\tau,0].
\end{aligned}
\end{align}
where $u\equiv u(x,t),\ \bar{u}\equiv u(x,t-\tau),$ and $\tau>0$.\\
Here $\mathcal{N}[u]=N\left[x,u,\dfrac{\partial  u}{\partial x},\dfrac{\partial^2 u}{\partial x^2},\dots,\dfrac{\partial^k u}{\partial x^k}\right]$ denotes a non-linear differential operator of order $k\ (k\in \mathbb{N})$, and $\dfrac{\partial^{\alpha}(\cdot)}{\partial t^{\alpha}}$ is a time-fractional derivative in the Caputo sense.\\
We define the linear space
\begin{eqnarray}
\begin{aligned}\label{lin}
\mathcal{W}_{n}=&\left\{z \ \Big| \mathcal{L}[z]=a_n\dfrac{d^nz}{dx^n}+a_{n-1} \dfrac{d^{n-1}z}{dx^{n-1}}+\dots+a_1\dfrac{dz}{dx}+a_{0} z=0,a_i\in\mathbb{R}, n\in \mathbb{N}\right\}\\
=&\ \text{Span}\left\{\varphi_1(x),\dots,\varphi_n(x)\right\},
\end{aligned}
\end{eqnarray}
where $\varphi_{1}(x),\dots,\varphi_{n}(x)$ form a solution set for some linear ordinary differential equation (ODE) of order $n$.\\
Corresponding to non-linear operator $\mathcal{H}[u,\bar{u}],$ the vector space $\mathcal{W}_n$ is invariant if $\mathcal{H}[\mathcal{W}_n,\mathcal{W}_n]\subseteq \mathcal{W}_n$, \textit{i.e.,} $\mathcal{H}[u,\bar{u}]\in \mathcal{W}_n$, for all $u,\bar{u}\in \mathcal{W}_n$.
If $\mathcal{W}_n$ is an invariant space corresponding to operator $\mathcal{H}[u,\bar{u}]$, then the invariant condition of $\mathcal{H}[u,\bar{u}]$  reduces to
\begin{equation}
\mathcal{L}\left(\mathcal{H}[u,\bar{u}]\right)\Big|_{\mathcal{L}(u)=0}=a_{n}\dfrac{d^n\mathcal{H}}{dx^n}+a_{n-1}\dfrac{d^{n-1}\mathcal{H}}{dx^{n-1}}+\dots+a_{1}\dfrac{d\mathcal{H}}{dx}+a_0\mathcal{H}\Big|_{\mathcal{L}(u)=0}=0,\ n\in\mathbb{N},
\end{equation}
where the constants $a_{n},\dots,a_0$ are to be determined.
Thus, there exist $n$-functions $\Theta_{1}$, $\Theta_{2}$,$\dots$,$\Theta_{n}$ such that
\begin{equation*}
\mathcal{H}\left[\sum^{n}_{i=1}A_{i}\varphi_{i}(x),\sum^{n}_{i=1}\bar{A}_{i}\varphi_{i}(x)\right]=\sum^{n}_{i=1}\Theta_{i}\left(A_{1},A_{2},\dots,A_{n}\right)\varphi_i(x)+\delta\sum\limits^{n}_{i=1}\bar{A}_i\varphi_{i}(x),\  \end{equation*}
where $A_{i}$ and  $\bar{A}_{i}$ $(i=1,2,\dots,n)$ are arbitrary real constants. Here $\left\{\Theta_i\right\}'$s are known as expansion coefficients of $\mathcal{H}[u,\bar{u}]\in\mathcal{W}_n$ with respect to the basis functions $\left\{\varphi_{i}\right\}'$s.\\
It follows that the time-fractional PDE with linear delay \eqref{gel} has a solution of the form
\begin{equation}
u(x,t)=\sum^{n}_{i=1}A_{i}(t)\varphi_{i}(x),
\end{equation}
where the coefficients $A_1(t), A_2(t),\dots, A_n(t)$ satisfy the following system of fractional delay ODEs
\begin{equation}\label{la}
\dfrac{d^{\alpha}A_i(t)}{dt^{\alpha}}=\Theta_{i}\left(A_{1}(t),A_{2}(t),\dots,A_{n}(t)\right)+\delta A_i(t-\tau),\ i=1,2,\dots,n.
\end{equation}
The fractional delay ODEs \eqref{la} are comparatively simple to handle.\\
\noindent\textbf{Note:} Using the invariant subspaces, the given time-fractional delay PDEs reduces to the system of fractional delay ODEs.
\section{Classification of invariant subspaces for time-fractional RD equations with time delay}
Here, we present an algorithmic approach to find invariant subspaces for the following equations:
\begin{itemize}
\item[(i)] Generalized time-fractional reaction-diffusion equation with time delay \eqref{crd}.
\item[(ii)]Time-fractional heat equation with source term involving time delay \eqref{qrd}.
\end{itemize}
\subsection{Generalized time-fractional RD equation involving linear time delay}
Generalized time-fractional reaction-diffusion equation involving a linear term with time delay \eqref{crd} can be written as
\begin{align*}\label{crd1}
\begin{aligned}
&\dfrac{\partial^{\alpha}u}{\partial t^{\alpha}}=\mathcal{H}_1[u,\bar{u}]\equiv D(u)u_{xx}+D_{u}(u)(u_{x})^2+R(u,\bar{u}),\ t>0,\  \alpha\in(0,1],\\
&u(x,t)=\Phi(x,t),\ t\in [-\tau,0],
\end{aligned}
\end{align*}
where the terms $D(u)$ and $R(u,\bar{u})$ denote the diffusion, and reaction with time delay ($\tau>0$), respectively, and $x\in \mathbb{R}$.\\

\noindent\textbf{An algorithmic approach for finding invariant subspaces:}
It may be noted that when $D(u)$ and $R(u,\bar{u})$ are arbitrary, there
exists no invariant subspace for the above Eq. \eqref{crd}. In this section, we consider a linear term incorporated with time delay, \textit{i.e.,} $R[u,\bar{u}]=M(u)+\delta\bar{u}$, where $M(u)$ is an arbitrary function of $u$.
Invariant subspace dimension theorem implies that the possible dimension of invariant subspaces corresponding to operator $\mathcal{H}_1$, is $n\leq 2k+1=1,2,3,4,5$ as order $k=2$ (\textit{cf.} \cite{gs}).\\
Consider the more general five-dimensional linear space
 \begin{equation}\label{in1}
 \begin{aligned}
 \mathcal{W}_{5}=&\left\{z \ \mid \mathcal{L}[z]=a_5z^{(v)}+a_4z^{(iv)}+a_3z{'''}+a_{2} z{''}+a_1z'+a_{0} z=0\right\}\\
 =&\ \text{Span} \left\{\varphi_1(x),\varphi_2(x),\varphi_3(x),\varphi_{4}(x),\varphi_{5}(x)\right\},
 \end{aligned}
 \end{equation}
 where $z^{(k)}=\dfrac{d^kz}{dx^k},\ k=1,\dots,5,$ and $\varphi_{1}(x),\dots,\varphi_{5}(x)$ form a solution set for some linear ODE of order $5$. Thus, the invariant condition of $\mathcal{H}_1[u,\bar{u}]$ takes the form
\begin{equation}\label{inv}
\mathcal{L}\left(\mathcal{H}_1[u,\bar{u}]\right)\Big|_{\mathcal{L}(u)=0}=a_5\dfrac{d^5\mathcal{H}_1}{dx^5}+a_{4}\dfrac{d^{4}\mathcal{H}_1}{dx^4}+\dots+a_{1}\dfrac{d\mathcal{H}_1}{dx}+a_0\mathcal{H}_1\Big|_{\mathcal{L}(u)=0}=0,
\end{equation}
where $\mathcal{H}_{1}[u,\bar{u}]=D(u)u_{xx}+D_{u}(u)(u_{x})^2+M(u)+\delta\bar{u},$ and the constants $a_{5},\dots,a_0$ are to be determined.
Simplifying equation \eqref{inv}, we get
\begin{align*}
&(-20a_1D_{uu}+a_3M_{uuu})(u_x)^3+(-6a_0D_u+a_2M_{uu}+a_1a_4D_u)(u_x)^2+210D_{uuu}(u_x)^2u_{xx}u_{xxx}\nonumber\\
&+105D_{uu}u_xu_{xx}u_{xxx}-6a_4D_u u_{xx}u_{xxxx}+(-11a_3D_u+10M_{uu})u_{xx}u_{xxx}+(-2a_3D_u+(a_4)^2\nonumber\\
&\times D_u+5M_{uu})u_xu_{xxxx}+(-3a_2D_u+4a_4M_uu+a_3a_4D_u)u_xu_{xxx}+(-25a_1D_u+3a_3M_{uu}+\nonumber\\
&a_2a_4D_{u})u_x u_{xx}-21a_0D_{u}u_{xx}u-6a_4D_{uu}(u_x)^2 u_{xxxx}+(-11a_3D_{uu}+10M_{uuu})(u_x)^2u_{xxx}+\nonumber\\
&(-15a_2D_{uu}+6a_4M_{uuu})(u_x)^{2}u_{xx}+(10a_3D_{uuu}+10M_{uuuu})(u_x)^3u_{xx}+(15a_3D_{uu}+15M_{uuu})\nonumber\\
&\times u_x(u_{xx})^2+45a_4D_{uuu} (u_x)^2(u_{xx})^2+15a_4D_{uuuu}(u_x)^4u_{xx}+20a_4D_{uuu}(u_x)^3u_{xxx}+60a_4\times\nonumber\\
& D_{uu}u_{x}u_{xx}u_{xxx}+15a_4D_{uu}(u_{xx})^3+3a_4M_{uu}(u_{xx})^2+10a_4D_{u}(u_{xxx})^2+(a_2D_{uuu}+a_4M_{uuuu})\nonumber\\
&\times (u_x)^4+(a_3D_{uuuu}+M_{uuuuu})(u_x)^5+a_4D_{uuuuu}(u_x)^6+105D_{uuu}u_{x}(u_{xx})^3+105D_{uuuu}u_x^3\nonumber\\
&\times (u_{xx})^2+105D_{uu}(u_{xx})^2u_{xxx}+70D_{uu}u_{x}(u_{xxx})^2-18a_2D_{u}(u_{xx})^2+35D_{uuu}(u_{x})^3u_{xxxx}\nonumber\\
&+21 D_{uuuuu}(u_x)^5u_{xx}+35D_{uuuu}(u_x)^4u_{xxx}+a_0M_{u}u+35D_{u}u_{xxx}u_{xxxx}+D_{uuuuu}(u_x)^7-21\nonumber\\
&\times a_0\ D_{uu}(u_x)^2u+a_0a_4D_{u}u_xu+a_0M=0.
\end{align*}
Simplification of the above equation gives an over-determined system, solving which we get different values of $D(u), M(u)$ and $a_i'$s. Corresponding to distinct $D(u), M(u),$ and $a_0$, $a_1$, $a_2$, $a_3$, $a_4$ and $a_{5}$ we find all possible operators $\mathcal{H}_1$, and their invariant subspaces of various dimensions as discussed below.
 \subsubsection{Invariant subspaces and corresponding  non-linear differential operators}
Here we present exponential, polynomial and trigonometric subspaces of the dimension $n=1,2,3,4,5$ along with the corresponding non-linear operators for \eqref{crd} .

\noindent\textbf{Exponential subspaces}\\
\textbf{1-D subspace:} 
For $a_5=a_4=a_3=a_2=0, a_1=1$, and $a_0\in \mathbb{R}$ in \eqref{in1}, the one-dimensional exponential subspace $\mathcal{W}_1=\left\{z \ \mid \mathcal{L}[z]=z'+a_{0} z=0\right\}=\text{Span}\left\{e^{-a_0 x}\right\}$ is invariant space for the operator
  \begin{align*}
  \mathcal{H}_1[u,\bar{u}]=&\left(b_nu^n+b_{n-1}u^{n-1}+\dots+b_1u+b_0\right)u_{xx}\\
&+\left(nb_nu^{n-1}+(n-1)b_{n-1}u^{n-2}+\dots+2b_2u+b_1\right)\left(u_x\right)^2\\
&+c_{n+1}u^{n+1}+c_{n}u^{n}+\dots+c_1u+\delta\bar{u},\ b_i,c_{i+1},\delta\in\mathbb{R},\ i=0,\dots,n,\ n\in\mathbb{N},
\end{align*}
if $c_{k+1}=-(k+1)a_0^2b_k$, $k=1,2,\dots,n$, $n\in\mathbb{N}$.

\noindent\textbf{2-D subspace:} 
For $a_5=a_4=a_3=a_0=0$ and $a_2=1$ in \eqref{in1}, we get the two-dimensional exponential subspace $\mathcal{W}_2=\left\{z \ \mid \mathcal{L}[z]= z{''}+a_1z'=0\right\}=\text{Span}\left\{1, e^{-a_1 x}\right\}$. $\mathcal{W}_2$ is invariant under $\mathcal{H}_1[u,\bar{u}]$ if
$D(u)=b_1 u+b_0$ and $R(u,\bar{u})=-2a_1^2b_1 u^2+c_1u+\delta\bar{u}+c_0$, $a_1,b_1,b_0,c_1,c_0,\delta\in\mathbb{R}$.\\


\noindent\textbf{Polynomial subspaces}\\
\noindent\textbf{2-D subspace:} 
When $a_0=a_1=a_3=a_4=a_5=0$ and $a_2\in\mathbb{R}$ in \eqref{in1}, we observe that the two-dimensional polynomial subspace $\mathcal{W}_2=\left\{z \ \mid \mathcal{L}[z]= a_2z{''}=0\right\}=\text{Span}\left\{1, x\right\}$ is invariant space with respect to $\mathcal{H}_1[u,\bar{u}]$ for the following choices of $D(u)$ and $R(u,\bar{u})$:\\
(i) $D(u)=b_1 u+b_0$ and $R(u,\bar{u})=c_1u+\delta\bar{u}+c_0$.\\
(ii) $D(u)=b_2u^2+b_1 u+b_0$ and $R(u,\bar{u})=c_1u+\delta\bar{u}+c_0$.

\noindent\textbf{3-D subspace:} 
Let $a_0=a_1=a_2=a_4=a_5=0$ and $a_3\in\mathbb{R}$ in \eqref{in1}.
Then the three-dimensional polynomial subspace $\mathcal{W}_3=\left\{z \ \mid \mathcal{L}[z]= a_3z{'''}=0\right\}=\text{Span}\left\{1, x, x^2\right\}$ is invariant corresponding to the differential operator $\mathcal{H}_1[u,\bar{u}]$ where $D(u)=b_1 u+b_0$ and $R(u,\bar{u})=c_1u+\delta\bar{u}+c_0$.

%

\begin{nt}
	It should be noted that when $D(u)=u^{-\frac{3}{2}}$ and $R(u,\bar{u})=c_1u+\delta\bar{u}+c_0$, the generalized time-fractional reaction-diffusion equation with linear time delay \eqref{crd} reduces to time-fractional fast diffusion equation involving a linear term with time delay
\begin{equation} \label{gs3}
\dfrac{\partial^{\alpha}u}{\partial t^{\alpha}}=\left(u^{-\frac{3}{2}}u_{x}\right)_{x}+c_1u+\delta\bar{u}+c_0.
\end{equation}
	When $\alpha=1,$ Eq. \eqref{gs3} admits a polynomial invariant subspace
$\mathcal{W}_{4}=\mathfrak{L}\left\{1,x,x^2,x^3\right\}$ of dimension four.
For $\alpha=1$, and $c_1=c_0=\delta=0$, Eq. \eqref{gs3} was studied by Galaktionov and Svirshchevskii \cite{gs}, and they found its exact solution in the polynomial subspace $\mathcal{W}_{4}$.
\end{nt}
\begin{nt}
When $D(u)=u^{-\frac{4}{3}}$ and $R(u,\bar{u})=u^{\frac{7}{3}}+c_1u+\delta\bar{u}+c_0$, the generalized time-fractional reaction-diffusion equation with linear time delay \eqref{crd} becomes
\begin{equation} \label{gs2}
\dfrac{\partial^{\alpha}u}{\partial t^{\alpha}}=\left(u^{-\frac{4}{3}}u_{x}\right)_{x}+u^{\frac{7}{3}}+c_1u+\delta\bar{u}+c_0.
\end{equation}Eq. \eqref{gs2} admits a polynomial invariant subspace
$\mathcal{W}_{5}=\mathfrak{L}\left\{1,x,x^2,x^3,x^4\right\}$ when $\alpha=1$.
It may further be noted that for $\alpha=1$, and $c_1=c_0=\delta_1=0$, the equation \eqref{gs2} was discussed and its exact solution in 5-dimensional polynomial space was found by Galaktionov and Svirshchevskii \cite{gs}.
\end{nt}

\noindent\textbf{Trigonometric subspaces}\\
\noindent\textbf{2-D subspace:} 
In this case, we assume that $a_5=a_4=a_3=a_1=0$, $a_2=1$ and $a_0\in\mathbb{R}$ in \eqref{in1}.
Thus, the two-dimensional trigonometric space is $\mathcal{W}_2=\left\{z \ \mid \mathcal{L}[z]= z{''}+a_0z=0\right\}=\text{Span}\left\{\cos{(\sqrt{a_0}x)},\sin{(\sqrt{a_0}x)}\right\}$. $\mathcal{H}_1[u,\bar{u}]$ admits invariant subspace $\mathcal{W}_2$ if $D(u)=b_2u^2+b_0$ and $R(u,\bar{u})=3a_0b_2u^3+c_1u+\delta\bar{u}$.

\noindent\textbf{3-D subspace:} 
For $a_5=a_4=a_2=a_0=0$, $a_1\in\mathbb{R}$ and $a_3=1$ in \eqref{in1}, the 3-dimensional space $\mathcal{W}_3=\left\{z \ \mid \mathcal{L}[z]= z{'''}+a_1z{'}=0\right\}=\text{Span}\left\{1,\cos{(\sqrt{a_1}x)},\sin{(\sqrt{a_1}x)}\right\}$ is invariant corresponding to operator $\mathcal{H}_1[u,\bar{u}]$, where $D(u)=b_1 u+b_0$ and $R(u,\bar{u})=2a_1b_1 u^2+c_1u+\delta\bar{u}+c_0$.

\begin{nt}
 For $D(u)=u^{-\frac{4}{3}}$ and $R(u,\bar{u})=-u^{-\frac{1}{3}}+c_1u+\delta\bar{u}+c_0$, the generalized reaction-diffusion equation with linear time delay \eqref{crd} reduces to the following time-fractional quasi-linear heat equation with linear time delay
\begin{equation} \label{gs1}
\dfrac{\partial^{\alpha}u}{\partial t^{\alpha}}=\left(u^{-\frac{4}{3}}u_{x}\right)_{x}-u^{-\frac{1}{3}}+c_1u+\delta\bar{u}+c_0.
\end{equation}
When $\alpha=1,$ Eq. \eqref{gs1} admits following invariant space of dimension five:
$$\mathcal{W}_{5}=\textrm{Span}\left\{1,\sin\left(\frac{2}{\sqrt{3}}x\right),\cos\left(\frac{2}{\sqrt{3}}x\right),\sin\left(\frac{4}{\sqrt{3}}x\right),\cos\left(\frac{4}{\sqrt{3}}x\right)\right\}.$$ For $\alpha=1$, $c_1=\delta_1=c_0=0$, time-fractional quasi-linear heat equation \eqref{gs1} was investigated and Compacton solutions with period $2\pi$ were derived by Galaktionov and Svirshchevskii \cite{gs}.
\end{nt}
\subsection{Time-fractional heat equation involving a source term with time delay}
In this section, we study the heat equation \eqref{qrd} involving a linear source term with time delay as follows
\begin{align*}\label{rdq}
&\dfrac{\partial^{\alpha}u}{\partial t^{\alpha}}=\mathcal{H}_2[u,\bar{u}]\equiv D(u)u_{xx}+R(u,\tilde{u}),\ t>0,\ \alpha\in(0,1],\\
&u(x,t)=\Psi(x,t),\ t\in [-\tau,0],\ \tau>0,
\end{align*}
where the terms $D(u)$ and $R(u,\bar{u})$ denote the diffusion and reaction with time delay respectively, and $x\in \mathbb{R}$.\\
Following the above procedure, we consider the general five-dimensional linear space
\begin{equation}\label{in2}
\begin{aligned}
 \mathcal{W}_{5}=&\left\{z \ \mid \mathcal{L}[z]=a_5z^{(v)}+a_4z^{(iv)}+a_3z{'''}+a_{2} z{''}+a_1z'+a_{0} z=0\right\}.
\end{aligned}
 \end{equation}
\subsubsection{Invariant subspaces and corresponding non-linear differential operators}
In this subsection, we classify exponential, polynomial and trigonometric subspaces of dimension $n=1,2,3,4,5$ with respect to the non-linear differential operators.\\
\noindent\textbf{Exponential subspaces}\\
\textbf{1-D subspace:} 
For $a_5=a_4=a_2=a_3=0$ and $a_1=1$ in \eqref{in2}, the obtained one-dimensional exponential subspace $\mathcal{W}_1=\left\{z \ \mid \mathcal{L}[z]=z'+a_{0} z=0\right\}=\text{Span}\left\{e^{-a_0 x}\right\}$ is a vector space which is invariant with respect to
  \begin{align*}
  \mathcal{H}_2[u,\bar{u}]=&\left(b_nu^n+b_{n-1}u^{n-1}+\dots+b_1u+b_0\right)u_{xx}\\
&+c_{n+1}u^{n+1}+c_{n}u^{n}+\dots+c_1u+\delta\bar{u},\ b_i,c_{i+1},\delta\in\mathbb{R},\ i=0,\dots,n,
\end{align*}
if $c_{k+1}=-a_0^2b_k$, $k\in\mathbb{N}$.

\noindent\textbf{2-D subspace:} 
If $a_5=a_4=a_3=a_0=0$ and $a_2=1$ in \eqref{in2}, then we get $\mathcal{W}_2=\left\{z \ \mid \mathcal{L}[z]= z{''}+a_1z'=0\right\}=\text{Span}\left\{1, e^{-a_1 x}\right\}$. This two-dimensional exponential subspace $\mathcal{W}_2$ is invariant corresponding to the operator $\mathcal{H}_2[u,\bar{u}],$ where
$D(u)=b_1 u+b_0$ and $R(u,\bar{u})=-a_1^2b_1 u^2+c_1u+\delta\bar{u}+c_0$.\\


\noindent\textbf{Polynomial subspaces}\\
\noindent\textbf{2-D subspace:} 
Parameters $a_0=a_1=a_3=a_4=a_5=0$ and $a_2\in\mathbb{R}$ in \eqref{in2}, leads to the two-dimensional polynomial subspace $\mathcal{W}_2=\left\{z \ \mid \mathcal{L}[z]= a_2z{''}=0\right\}=\text{Span}\left\{1, x\right\}$. Further note that $\mathcal{W}_2$ is invariant under $\mathcal{H}_2[u,\bar{u}]$ where $D(u)$ and $R(u,\bar{u})$ are as follows:\\
(i) $D(u)=b_1 u+b_0$ and $R(u,\bar{u})=c_1u+\delta\bar{u}+c_0$.\\
(ii) $D(u)=b_2u^2+b_1 u+b_0$ and $R(u,\bar{u})=c_1u+\delta\bar{u}+c_0$.\\
(iii) $D(u)=b_3u^3+b_2u^2+b_1 u+b_0$ and $R(u,\bar{u})=c_1u+\delta\bar{u}+c_0$.\\
(iv) $D(u)$-arbitrary and $R(u,\bar{u})=c_1u+\delta\bar{u}+c_0$.

\noindent\textbf{3-D subspace:} 
Let $a_0=a_1=a_2=a_4=a_5=0$ and $a_3\in\mathbb{R}$ in \eqref{in2}.
Then, the obtained three-dimensional polynomial subspace $\mathcal{W}_3=\left\{z \ \mid \mathcal{L}[z]= a_3z{'''}=0\right\}=\text{Span}\left\{1, x, x^2\right\}$ is invariant space admitted by $\mathcal{H}_2[u,\bar{u}]$ if $D(u)=b_1 u+b_0$ and $R(u,\bar{u})=c_1u+\delta\bar{u}+c_0$.\\

%

\noindent\textbf{Trigonometric subspaces}\\
\noindent\textbf{2-D subspace:} 
In this case, we assume $a_5=a_4=a_3=a_1=0$, $a_2=1$ and $a_0\in\mathbb{R}$ in \eqref{in2}.
Thus $\mathcal{H}_2[u,\bar{u}]$ admits the two-dimensional trigonometric invariant subspace $\mathcal{W}_2=\left\{z \ \mid \mathcal{L}[z]= z{''}+a_0z=0\right\}=\text{Span}\left\{\cos{(\sqrt{a_0}x)},\sin{(\sqrt{a_0}x)}\right\}$ for the following cases:\\
 (i) $D(u)=b_2u^2+b_1u+b_0$ and $R(u,\bar{u})=a_0b_2u^3+b_1a_0u^2+c_1u+\delta\bar{u}$.\\
(ii) $D(u)=b_1u+b_0$ and $R(u,\bar{u})=b_1a_0u^2+c_1u+\delta\bar{u}$.

\noindent\textbf{3-D subspace:} 
For $a_5=a_4=a_2=a_0=0$, $a_1\in\mathbb{R}$ and $a_3=1$ in \eqref{in2}, the three-dimensional trigonometric invariant subspace is $\mathcal{W}_3=\left\{z \ \mid \mathcal{L}[z]= z{'''}+a_1z{'}=0\right\}=\text{Span}\left\{1,\cos{(\sqrt{a_0}x)},\sin{(\sqrt{a_0}x)}\right\}$ which is admitted by the differential operator\\ $\mathcal{H}_2[u,\bar{u}]=(b_1 u+b_0)u_{xx}+b_1a_1 u^2+c_1u+\delta\bar{u}+c_0$.

%
\subsection{Invariant subspaces corresponding to linear differential operators $\mathcal{H}_{1}[u,\bar{u}]$ and $\mathcal{H}_{2}[u,\bar{u}]$}
In this subsection, we discuss the invariant subspaces corresponding to the linear differential operators $\mathcal{H}_1[u,\bar{u}]$ and $\mathcal{H}_2[u,\bar{u}]$ for the time-fractional delay reaction-diffusion equations \eqref{crd} and \eqref{qrd}, respectively.\\
\textbf{Case (i):} If $a_0=0$ and $a_i\in \mathbb{R},\ i=1,\dots,5,$ then  the linear differential operators
\begin{equation*}
\mathcal{H}_1[u,\bar{u}]=\mathcal{H}_2[u,\bar{u}]=b_0u_{xx}+c_1u+\delta\bar{u}+c_0,
\end{equation*}
where $D(u)=b_0$ and $R(u,\bar{u})=c_1u+\delta\bar{u}+c_0$, admit the following invariant subspaces:
\begin{itemize}
\item[1.] $\mathcal{W}_{n+1}=\text{Span}\left\{1,x,\dots,x^{n}\right\}$.
\item[2.] $\mathcal{W}_{n+1}=\text{Span}\left\{1,e^{\nu_1x},\dots,e^{\nu_n x}\right\}$.
\item[3.] $\mathcal{W}_{n+1}=\text{Span}\left\{1,\cos{(\kappa_1x)},\sin{(\omega_1x)},\dots,\cos{(\kappa_{\frac{n}{2}} x)},\sin{(\omega_{\frac{n}{2}} x)}\right\}$.
\item[4.] $\mathcal{W}_{2n+1}=\text{Span}\left\{1,x,\dots,x^{n},e^{\nu_1x},\dots,e^{\nu_n x}\right\}$.
\item[5.] $\mathcal{W}_{2n+1}=\text{Span}\left\{1,x,\dots,x^{n},\cos{(\kappa_1x)},\sin{(\omega_1x)},\dots,\cos{(\kappa_{\frac{n}{2}} x)},\sin{(\omega_{\frac{n}{2}} x)}\right\}$.
\item[6.] $\mathcal{W}_{2n+1}=\text{Span}\left\{1,e^{\nu_1x},\dots,e^{\nu_n x},\cos{(\kappa_1x)},\sin{(\omega_1x)},\dots,\cos{(\kappa_{\frac{n}{2}} x)},\sin{(\omega_{\frac{n}{2}} x)}\right\}$.
\item[7.] $\mathcal{W}_{3n+1}=\text{Span}\left\{1,x,\dots,x^{n},e^{\nu_1x},\dots,e^{\nu_n x},\cos{(\kappa_1x)},\sin{(\omega_1x)},\dots,\cos{(\kappa_{\frac{n}{2}} x)},\sin{(\omega_{\frac{n}{2}} x)}\right\}$.
\item[8.]  $\mathcal{W}_{n+1}=\text{Span}\left\{1,e^{\mu_1 x}\cos{(\kappa_1x)},e^{\mu_{1}x}\sin{(\omega_1x)},\dots,e^{\mu_{\frac{n}{2}}x}\cos{(\kappa_{\frac{n}{2}} x)},e^{\mu_{\frac{n}{2}}x}\sin{(\omega_{\frac{n}{2}} x)}\right\}$.
\item[9.] $\mathcal{W}_{2n+1}=\text{Span}\left\{1,x,\dots,x^{n},e^{\mu_1 x}\cos{(\kappa_1x)},e^{\mu_{1}x}\sin{(\omega_1x)},\dots,e^{\mu_{\frac{n}{2}}x}\cos{(\kappa_{\frac{n}{2}} x)},e^{\mu_{\frac{n}{2}}x}\sin{(\omega_{\frac{n}{2}} x)}\right\}$.
\item[10.] $\mathcal{W}_{3n+1}=\text{Span}\left\{1,x,\dots,x^{n},e^{\nu_1x},\dots,e^{\nu_n x},e^{\mu_1 x}\cos{(\kappa_1x)},e^{\mu_{1}x}\sin{(\omega_1x)},\dots,\right.$\\
    $\left.\qquad\qquad\qquad \times\  e^{\mu_{\frac{n}{2}}x}\cos{(\kappa_{\frac{n}{2}} x)},e^{\mu_{\frac{n}{2}}x}\sin{(\omega_{\frac{n}{2}} x)}\right\}$.
\end{itemize}Here $n\in \mathbb{N}$ and $\nu_{i}, \mu_{i},\kappa_{i},\omega_{i}\in\mathbb{R}$.\\
\textbf{Case (ii):} If $a_0\neq0$ and $a_i\in \mathbb{R},\ i=1,\dots,5,$ we obtain $D(u)=b_0$ and $R(u,\bar{u})=c_1u+\delta\bar{u}$. The corresponding differential operators $\mathcal{H}_{1}[u,\bar{u}]$ and $\mathcal{H}_{2}[u,\bar{u}]$ admit the following invariant subspaces:
\begin{itemize}
\item[1.] $\mathcal{W}_{n}=\text{Span}\left\{x,\dots,x^{n}\right\}$.
\item[2.] $\mathcal{W}_{n}=\text{Span}\left\{e^{\nu_1x},\dots,e^{\nu_n x}\right\}$.
\item[3.] $\mathcal{W}_{n}=\text{Span}\left\{\cos{(\kappa_1x)},\sin{(\omega_1x)},\dots,\cos{(\kappa_{\frac{n}{2}} x)},\sin{(\omega_{\frac{n}{2}} x)}\right\}$.
\item[4.] $\mathcal{W}_{2n}=\text{Span}\left\{x,\dots,x^{n},e^{\nu_1x},\dots,e^{\nu_n x}\right\}$.
\item[5.] $\mathcal{W}_{2n}=\text{Span}\left\{x,\dots,x^{n},\cos{(\kappa_1x)},\sin{(\omega_1x)},\dots,\cos{(\kappa_{\frac{n}{2}} x)},\sin{(\omega_{\frac{n}{2}} x)}\right\}$.
\item[6.] $\mathcal{W}_{2n}=\text{Span}\left\{e^{\nu_1x},\dots,e^{\nu_n x},\cos{(\kappa_1x)},\sin{(\omega_1x)},\dots,\cos{(\kappa_{\frac{n}{2}} x)},\sin{(\omega_{\frac{n}{2}} x)}\right\}$.
\item[7.] $\mathcal{W}_{3n}=\text{Span}\left\{x,\dots,x^{n},e^{\nu_1x},\dots,e^{\nu_n x},\cos{(\kappa_1x)},\sin{(\omega_1x)},\dots,\cos{(\kappa_{\frac{n}{2}} x)},\sin{(\omega_{\frac{n}{2}} x)}\right\}$.
\item[8.]  $\mathcal{W}_{n}=\text{Span}\left\{e^{\mu_1}\cos{(\kappa_1x)},e^{\mu_{1}}\sin{(\omega_1x)},\dots,e^{\mu_{\frac{n}{2}}}\cos{(\kappa_{\frac{n}{2}} x)},e^{\mu_{\frac{n}{2}}}\sin{(\omega_{\frac{n}{2}} x)}\right\}$.
\item[9.] $\mathcal{W}_{2n}=\text{Span}\left\{x,\dots,x^{n},e^{\mu_1}\cos{(\kappa_1x)},e^{\mu_{1}}\sin{(\omega_1x)},\dots,e^{\mu_{\frac{n}{2}}}\cos{(\kappa_{\frac{n}{2}} x)},e^{\mu_{\frac{n}{2}}}\sin{(\omega_{\frac{n}{2}} x)}\right\}$.
\item[10.] $\mathcal{W}_{3n}=\text{Span}\left\{x,\dots,x^{n},e^{\nu_1x},\dots,e^{\nu_n x},e^{\mu_1 x}\cos{(\kappa_1x)},e^{\mu_{1}x}\sin{(\omega_1x)},\dots,e^{\mu_{\frac{n}{2}}x}\cos{(\kappa_{\frac{n}{2}} x)},\right.$
    $\left.\qquad\qquad\quad \times\ e^{\mu_{\frac{n}{2}}x}\sin{(\omega_{\frac{n}{2}} x)}\right\}$.
\end{itemize}Here $n\in \mathbb{N}$ and $\nu_{i}, \mu_{i},\kappa_{i},\omega_{i}\in\mathbb{R}$.\\
\section{Exact solutions for time-fractional RD equations with linear time delay}
\subsection{Construction of exact solutions for \eqref{crd}:}
\subsubsection{One-dimensional exponential solution}
Let $D(u)=b_nu^n+b_{n-1}u^{n-1}+\dots+b_1u+b_0$ and $R(u,\bar{u})=c_{n+1}u^{n+1}+c_nu^n+\dots+c_1u+\delta\bar{u}$. Thus, the time-fractional reaction-diffusion equation involving a linear term with time delay \eqref{crd} reduces to
\begin{eqnarray}
\begin{aligned}\label{crd2}
\dfrac{\partial^{\alpha}u}{\partial t^{\alpha}}=&\mathcal{H}_1[u,\bar{u}]=\left(b_nu^n+b_{n-1}u^{n-1}+\dots+b_1u+b_0\right)u_{xx}\\
&+\left(nb_nu^{n-1}+(n-1)b_{n-1}u^{n-2}+\dots+2b_2u+b_1\right)\left(u_x\right)^2\\
&+c_{n+1}u^{n+1}+c_{n}u^{n}+\dots+c_1u+\delta\bar{u},\ t>0,\ b_i,c_{i+1},\delta\in\mathbb{R},i=0,\dots,n,\ n\in\mathbb{N},
\end{aligned}
\end{eqnarray}along with the initial condition
\begin{equation}\label{4.2}
u=\Phi(x,t)=\psi(t)e^{a_0x},\ t\in [-\tau,0].
\end{equation}
 Let $\mathcal{W}_{1}=\text{Span}\left\{e^{a_0x}\right\}$, $a_0\in\mathbb{R}$. The linear exponential space $\mathcal{W}_{1}$ is invariant corresponding to $\mathcal{H}_1[u,\bar{u}]$ if $c_{k+1}=-(k+1)a_0^2b_i$, $k=1,2,\dots,n$, as
$\mathcal{H}_1\left[Ae^{a_0x}, \bar{A}e^{a_0x}\right]=\left(a_0^2b_0+c_1\right)Ae^{a_0x}+\delta\bar{A}e^{a_0x}\in\mathcal{W}_1$.
Thus an exact solution of time-fractional PDE \eqref{crd2} is of the form
\begin{equation}\label{crd3}
u(x,t)=A(t)e^{a_0x},\ a_0\in\mathbb{R},
\end{equation}
where unknown function $A(t)$ is to be determined by solving
\begin{equation}\label{l1}
\dfrac{d^{\alpha} A}{dt^{\alpha}}=\left(a_0^{2}b_{0}+c_1\right)A(t)+\delta A(t-\tau).
\end{equation}
Here the initial condition \eqref{4.2} implies that $A(t)=\psi(t)$, $t\in[-\tau,0]$.\\
Applying the Laplace transformation on both sides of equation \eqref{l1}, we have
\begin{align*}
s^{\alpha}\hat{A}(s)-s^{\alpha-1}\psi(0)&=\left(a_0^{2}b_{0}+c_1\right)\hat{A}(s)+\delta\mathrm{L}\left\{A(t-\tau)\right\},\\
s^{\alpha}\hat{A}(s)-s^{\alpha-1}\psi(0)&=\left(a_0^{2}b_{0}+c_1\right)\hat{A}(s)+\delta e^{-\tau s}\int\limits_{-\tau}^0 e^{-s\mu}\psi(\mu)d\mu+\delta e^{-\tau s}\hat{A}(s).
\end{align*}
By simplification, we obtain
\begin{equation}\label{l2}
\hat{A}(s)=\psi(0)\left(\dfrac{s^{\alpha-1}}{s^\alpha-\lambda-\delta e^{-\tau s}}\right)+\delta\left(\dfrac{e^{-\tau s}}{s^\alpha-\lambda-\delta e^{-\tau s}}\right)\int\limits_{-\tau}^0 e^{-s\mu}\psi(\mu)d\mu,\ \lambda=a_0^{2}b_{0}+c_1.
\end{equation}Taking inverse Laplace transformation and using convolution theorem in \eqref{l2}, we get
\begin{equation*}
A(t)=\psi(0)\mathrm{L}^{-1}\left\{\dfrac{s^{\alpha-1}}{s^\alpha-\lambda-\delta e^{-\tau s}}\right\}+\delta\mathrm{L}^{-1}\left\{\dfrac{1}{s^\alpha-\lambda-\delta e^{-\tau s}}\right\}\star\mathrm{L}^{-1}\left\{e^{-\tau s}\int\limits_{-\tau}^0 e^{-s\mu}\psi(\mu)d\mu\right\}.
\end{equation*}Consider
\begin{align*}
\mathrm{L}^{-1}\left\{\dfrac{s^{\alpha-1}}{s^\alpha-\lambda-\delta e^{-\tau s}}\right\}=&\mathrm{L}^{-1}\left\{\dfrac{s^{\alpha-1}}{\left(1-\dfrac{\delta e^{-s\tau}}{s^{\alpha}-\lambda}\right)(s^{\alpha}-\lambda)}\right\}\\
=&\mathrm{L}^{-1}\left\{\sum\limits_{n=0}^{\infty}\dfrac{\delta^n e^{-\tau ns}s^{\alpha-1}}{(s^{\alpha}-\lambda)^{n+1}}\right\},\ \Bigg| \frac{\delta e^{-s\tau}}{s^{\alpha}-\lambda}\Bigg|<1,\\
=&\sum\limits_{n=0}^{\infty}\delta^n H(t-n\tau)(t-n\tau)^{\alpha n}E_{\alpha,\alpha n+1}^{n+1}(\lambda(t-n\tau)^{\alpha}),
\end{align*}
where $H(t-n\tau)$ is a delayed unit step function.
Further we note
\begin{align*}
\mathrm{L}^{-1}\left\{\dfrac{1}{s^\alpha-\lambda-\delta e^{-\tau s}}\right\}=&\sum\limits_{n=0}^{\infty}\delta^n\mathrm{L}^{-1}\left\{\dfrac{ e^{-\tau ns}}{(s^{\alpha}-\lambda)^{n+1}}\right\},\  \Bigg| \frac{\delta  e^{-s\tau}}{s^{\alpha}-\lambda}\Bigg|<1,\\
=&\sum\limits_{n=0}^{\infty}\delta^n H(t-n\tau)(t-n\tau)^{\alpha n+\alpha-1}E_{\alpha,\alpha n+\alpha}^{n+1}(\lambda(t-n\tau)^{\alpha}).
\end{align*}
Finally, we compute $\mathrm{L}^{-1}\left\{e^{-\tau s}\int\limits_{-\tau}^0 e^{-s\mu}\psi(\mu)d\mu\right\}$.
Define $g(t):[-\tau, \infty)\mapsto[0,1]$ by
$$g(t)=\left\{
              \begin{array}{ll}
                0, & \hbox{if}\ t\geq0; \\
                1, & \hbox{if}\ t<0.
              \end{array}
            \right.
$$
Extending $\psi(t)$ from $[-\tau,0)$ to $[-\tau,\infty)$ by defining $\psi(t)=\psi(0)$ for $t\geq0$, then
\begin{align*}
\mathrm{L}^{-1}\left\{e^{-\tau s}\int\limits_{-\tau}^0 e^{-s\mu}\psi(\mu)d\mu\right\}
=&\mathrm{L}^{-1}\left\{\int\limits_{0}^\infty e^{-s\xi}\psi(-\tau+\xi)g(-\tau+\xi)d\xi\right\}\\
=&\mathrm{L}^{-1}\left\{\mathrm{L}\left\{\psi(-\tau+\xi)g(-\tau+\xi)\right\}\right\}=\psi(t-\tau)g(t-\tau).
\end{align*}
Thus Eq. \eqref{l2} becomes
\begin{align*}
A(t)=&\psi(0)\sum\limits_{n=0}^{\infty}\delta^n H(t-n\tau)(t-n\tau)^{\alpha n}E_{\alpha,\alpha n+1}^{n+1}(\lambda(t-n\tau)^{\alpha})\\
&+\left[\sum\limits_{n=0}^{\infty}\delta^{n+1} H(t-n\tau)(t-n\tau)^{\alpha n+\alpha-1}E_{\alpha,\alpha n+\alpha}^{n+1}(\lambda(t-n\tau)^{\alpha})\right]\star\left[\psi(t-\tau)g(t-\tau)\right]\\
=&\psi(0)\sum\limits_{n=0}^{\floor*{\frac{t}{\tau}}}\delta^n (t-n\tau)^{\alpha n}E_{\alpha,\alpha n+1}^{n+1}(\lambda(t-n\tau)^{\alpha})\\
&+\left[\sum\limits_{n=0}^{\floor*{\frac{t}{\tau}}}\delta^{n+1}(t-n\tau)^{\alpha n+\alpha-1}E_{\alpha,\alpha n+\alpha}^{n+1}(\lambda(t-n\tau)^{\alpha})\right]\star\left[\psi(t-\tau)g(t-\tau)\right].
\end{align*}
Hence exact solution of generalized time-fractional RD equation with time delay \eqref{crd2} corresponding to 1-dimensional exponential subspace is
\begin{align*}
u(x,t)=\ &\psi(0)\left[\sum\limits_{n=0}^{\floor*{\frac{t}{\tau}}}\delta^n (t-n\tau)^{\alpha n}E_{\alpha,\alpha n+1}^{n+1}(\lambda(t-n\tau)^{\alpha})+\sum\limits_{n=0}^{\floor*{\frac{t}{\tau}}}\delta^{n+1}\int\limits^t_0 (r-n\tau)^{\alpha n+\alpha-1}\right.\\
&\times E_{\alpha,\alpha n+\alpha}^{n+1}(\lambda(r-n\tau)^{\alpha})\psi(t-\tau-r)g(t-\tau-r)dr\Bigg]e^{a_0x},
\end{align*}
where $A(t)=\psi(t)$, $t\in[-\tau,0]$ and $\lambda=a_0^2b_0+c_1$, $a_0,b_0,c_1,\delta\in\mathbb{R}$.
\subsubsection{Two-dimensional polynomial solution}
Consider the polynomial subspace $\mathcal{W}_2=\text{Span}\left\{1,x\right\}$ along with $D(u)=b_0$ and $R(u,\bar{u})=c_1u+\delta\bar{u}+c_0$, $b_0,c_0,\delta,c_1\in\mathbb{R}$.\\
The initial condition here reads $u(x,t)=\Phi(x,t)=\psi(t)+\phi(t)x,\ -\tau\leq t\leq 0.$\\
Here, the polynomial exact solution for \eqref{crd} is
\begin{equation}
u(x,t)=A_1(t)+A_2(t)x,
\end{equation}
where $A_1(t)$ and $A_2(t)$ satisfy the following system of linear fractional delay ODEs
\begin{eqnarray}
&&\label{po1}\dfrac{d^{\alpha}A_1}{dt^{\alpha}}=c_1A_1(t)+\delta A_1(t-\tau)+c_0,\\
&&\label{po2}\dfrac{d^{\alpha}A_2}{dt^{\alpha}}=c_1A_2(t)+\delta A_2(t-\tau).
\end{eqnarray}
Applying the Laplace and inverse Laplace transform along with convolution theorem to \eqref{po1}, and proceeding the above similar procedure, we get
\begin{equation*}
\begin{aligned}
A_1(t)=&\psi(0)\mathrm{L}^{-1}\left\{\dfrac{s^{\alpha-1}}{s^\alpha-c_1-\delta e^{-\tau s}}\right\}+\mathrm{L}^{-1}\left\{\dfrac{c_0s^{-1}}{s^\alpha-c_1-\delta e^{-\tau s}}\right\}\\
&+\delta \mathrm{L}^{-1}\left\{\dfrac{1}{s^\alpha-c_1-\delta e^{-\tau s}}\right\}\star\mathrm{L}^{-1}\left\{e^{-\tau s}\int\limits_{-\tau}^0 e^{-s\mu}\psi(\mu)d\mu\right\}.
\end{aligned}
\end{equation*}
Consider
\begin{eqnarray}
\begin{aligned}
\mathrm{L}^{-1}\left\{\dfrac{c_0s^{-1}}{s^\alpha-c_1-\delta e^{-\tau s}}\right\}=&\mathrm{L}^{-1}\left\{\dfrac{c_0s^{-1}}{\left(1-\dfrac{\delta e^{-s\tau}}{s^{\alpha}-c_1}\right)(s^{\alpha}-c_1)}\right\},\ \Bigg| \dfrac{\delta e^{-s\tau}}{s^{\alpha}-c_1}\Bigg|<1, \\
=&\sum\limits_{n=0}^{\infty}c_0\delta^n H(t-n\tau)(t-n\tau)^{\alpha (n+1)}E_{\alpha,\alpha n+\alpha+1}^{n+1}(c_1(t-n\tau)^{\alpha}).
\end{aligned}
\end{eqnarray}
Thus, we obtain
\begin{align*}
A_1(t)=&\psi(0)\sum\limits_{n=0}^{\floor*{\frac{t}{\tau}}}\delta^n (t-n\tau)^{\alpha n}E_{\alpha,\alpha n+1}^{n+1}(c_1(t-n\tau)^{\alpha})\\
&+\left[\sum\limits_{n=0}^{\floor*{\frac{t}{\tau}}}\delta^{n+1} (t-n\tau)^{\alpha n+\alpha-1}E_{\alpha,\alpha n+\alpha}^{n+1}(c_1(t-n\tau)^{\alpha})\right]\star\left[\psi(t-\tau)g(t-\tau)\right]\\
&+\sum\limits_{n=0}^{\floor*{\frac{t}{\tau}}}c_0\delta^n (t-n\tau)^{\alpha (n+1)}E_{\alpha,\alpha n+\alpha+1}^{n+1}(c_1(t-n\tau)^{\alpha}).
\end{align*}
Similarly, we compute
\begin{align*}
A_2(t)=&\phi(0)\sum\limits_{n=0}^{\floor*{\frac{t}{\tau}}}\delta^n (t-n\tau)^{\alpha n}E_{\alpha,\alpha n+1}^{n+1}(c_1(t-n\tau)^{\alpha})\\
&+\left[\sum\limits_{n=0}^{\floor*{\frac{t}{\tau}}}\delta^{n+1}(t-n\tau)^{\alpha n+\alpha-1}E_{\alpha,\alpha n+\alpha}^{n+1}(c_1(t-n\tau)^{\alpha})\right]\star\left[\phi(t-\tau)g(t-\tau)\right].
\end{align*}
Hence, we obtain an exact solution of time-fractional reaction-diffusion equation with time delay \eqref{crd} along with $D(u)=b_0$ and $R(u,\bar{u})=c_1u+\delta\bar{u}+c_0$, as
\begin{align*}
u(x,t)=&\psi(0)\sum\limits_{n=0}^{\floor*{\frac{t}{\tau}}}\delta^n (t-n\tau)^{\alpha n}E_{\alpha,\alpha n+1}^{n+1}(c_1(t-n\tau)^{\alpha})+\sum\limits_{n=0}^{\floor*{\frac{t}{\tau}}}\delta^{n+1}\int\limits^t_0 (r-n\tau)^{\alpha n+\alpha-1}\\
&\times E_{\alpha,\alpha n+\alpha}^{n+1}(c_1(r-n\tau)^{\alpha})\psi(t-\tau-r)g(t-\tau-r)dr+\sum\limits_{n=0}^{\floor*{\frac{t}{\tau}}}c_0\delta^n (t-n\tau)^{\alpha (n+1)}\\
&\times E_{\alpha,\alpha n+\alpha+1}^{n+1}(c_1(t-n\tau)^{\alpha})+\left[\phi(0)\sum\limits_{n=0}^{\floor*{\frac{t}{\tau}}}\delta^n (t-n\tau)^{\alpha n}E_{\alpha,\alpha n+1}^{n+1}(c_1(t-n\tau)^{\alpha})\right.\\
&\left.+ \sum\limits_{n=0}^{\floor*{\frac{t}{\tau}}}\delta^{n+1}\int\limits^t_0 (r-n\tau)^{\alpha n+\alpha-1}E_{\alpha,\alpha n+\alpha}^{n+1}(c_1(r-n\tau)^{\alpha})\phi(t-\tau-r)g(t-\tau-r)dr\right]x,
\end{align*}
where $A_1(t)=\psi(t)$ and $A_2(t)=\phi(t)$, $t\in[-\tau,0]$.
\subsubsection{Three-dimensional trigonometric solution}
Consider the following invariant space $$\mathcal{W}_3=\ \text{Span}\left\{1,\cos{(\sqrt{a_1}x)},\sin{(\sqrt{a_1}x)}\right\}$$ admitted by $\mathcal{H}_1[u,\bar{u}]=b_0u_{xx}+c_1u+\delta\bar{u}+c_0$ along with the initial condition $u(x,t)=\psi(t)+\phi(t)\cos(\sqrt{a_1}x)+\eta(t)\sin(\sqrt{a_1}x)$. Following the above similar procedure, we find an analytical solution of \eqref{crd} corresponding to $\mathcal{H}_1[u,\bar{u}]$ as follows
\begin{align*}
u(x,t)&=A_1(t)+A_2(t)\cos{(\sqrt{a_1}x)}+A_3(t)\sin{(\sqrt{a_1}x)}\\
&=\psi(0)\sum\limits_{n=0}^{\floor*{\frac{t}{\tau}}}\delta^n (t-n\tau)^{\alpha n}E_{\alpha,\alpha n+1}^{n+1}(c_1(t-n\tau)^{\alpha})+\sum\limits_{n=0}^{\floor*{\frac{t}{\tau}}}\delta^{n+1}\int\limits^t_0 (r-n\tau)^{\alpha n+\alpha-1}\\
&\times E_{\alpha,\alpha n+\alpha}^{n+1}(c_1(r-n\tau)^{\alpha})\psi(t-\tau-r)g(t-\tau-r)dr+\sum\limits_{n=0}^{\floor*{\frac{t}{\tau}}}c_0\delta^n (t-n\tau)^{\alpha (n+1)}\\
&\times E_{\alpha,\alpha n+\alpha+1}^{n+1}(c_1(t-n\tau)^{\alpha})+\left[\phi(0)\sum\limits_{n=0}^{\floor*{\frac{t}{\tau}}}\delta^{n} (t-n\tau)^{\alpha n}E_{\alpha,\alpha n+1}^{n+1}(\gamma(t-n\tau)^{\alpha})\right.\\
&\left.+\sum\limits_{n=0}^{\floor*{\frac{t}{\tau}}}\delta^{n+1}\int\limits^t_0 (r-n\tau)^{\alpha n+\alpha-1}E_{\alpha,\alpha n+\alpha}^{n+1}(\gamma(r-n\tau)^{\alpha})\phi(t-\tau-r)g(t-\tau-r)dr\right]\\
&\times\cos{(\sqrt{a_1}x)}+\left[\eta(0)\sum\limits_{n=0}^{\floor*{\frac{t}{\tau}}}\delta^n (t-n\tau)^{\alpha n}E_{\alpha,\alpha n+1}^{n+1}(\gamma(t-n\tau)^{\alpha})+ \sum\limits_{n=0}^{\floor*{\frac{t}{\tau}}}\delta^{n+1}\right.\\
&\left.\int\limits^t_0 (r-n\tau)^{\alpha n+\alpha-1}E_{\alpha,\alpha n+\alpha}^{n+1}(\gamma(r-n\tau)^{\alpha})\eta(t-\tau-r)g(t-\tau-r)dr\right]\sin{(\sqrt{a_1}x)},
\end{align*}
where $A_1(t)=\psi(t)$, $A_2(t)=\phi(t)$, $A_3(t)=\eta(t)$, $t\in[-\tau,0]$, and $\gamma=c_1-a_1b_0$.

\subsection{Construction of exact solutions for \eqref{qrd}:}
\subsubsection{One-dimensional exponential solution}
Consider an exponential subspace $\mathcal{W}_{1}=\text{Span}\left\{e^{-a_0x}\right\}$ of dimension one along with $D(u)=b_nu^n+b_{n-1}u^{n-1}+\dots+b_1u+b_0$ and $R(u,\bar{u})=c_{n+1}u^{n+1}+c_{n}u^{n}+\dots+c_1u+\delta\bar{u}$. Thus, the time-fractional heat equation with linear term involving time delay \eqref{qrd} reduces to
\begin{eqnarray}\label{p0}
\begin{aligned}
\dfrac{\partial^{\alpha}u}{\partial t^{\alpha}}=\mathcal{H}_2[u,\bar{u}]=&\left(b_nu^n+b_{n-1}u^{n-1}+\dots+b_1u+b_0\right)\dfrac{\partial^2u}{\partial x^2}\\
&+c_{n+1}u^{n+1}+c_{n}u^{n}+\dots+c_1u+\delta\bar{u},\ t>0,\ \alpha\in(0,1].
\end{aligned}
\end{eqnarray}along with the initial condition $u(x,t)=\Psi(x,t)=e^{-a_0x}\xi(t),\, t\in[-\tau,0].$\\
The linear space $\mathcal{W}_1=\text{Span}\left\{e^{-a_0 x}\right\}$ is admitted by $\mathcal{H}_2[u,\bar{u}]$ if $c_{k+1}=-a_0^2b_k$, $k=1,\dots,m, m\in\mathbb{N}$.
Thus, we find an exact solution of \eqref{p0} corresponding to space $\mathcal{W}_1$ as
\begin{align*}
u(x,t)=A(t)e^{-a_0x}=&\left[\xi(0)\sum\limits_{n=0}^{\floor*{\frac{t}{\tau}}}\delta^n (t-n\tau)^{\alpha n}E_{\alpha,\alpha n+1}^{n+1}(\beta(t-n\tau)^{\alpha})+\sum\limits_{n=0}^{\floor*{\frac{t}{\tau}}}\delta^{n+1}\times\right.\\
&\int\limits^t_0 (r-n\tau)^{\alpha n+\alpha-1} E_{\alpha,\alpha n+\alpha}^{n+1}(\beta(r-n\tau)^{\alpha})\xi(t-\tau-r)g(t-\tau-r)dr\Bigg]e^{-a_0x},
\end{align*}
where $A(t)=\xi(t)$, $t\in[-\tau,0]$.
\subsubsection{Two-dimensional polynomial solution}
Consider $\mathcal{W}_1=\text{Span}\left\{1,x\right\}$ with respect to the following time-fractional heat equation with time delay \eqref{qrd}
\begin{align}\label{p1}
\begin{aligned}
&\dfrac{\partial^{\alpha}u}{\partial t^{\alpha}}=D(u)u_{xx}+c_1u+\delta\bar{u}+c_0,\ t>0,\ \alpha\in(0,1],\\
&u(x,t)=\Psi(x,t)=\psi(t)+\chi(t) x,\, t\in[-\tau,0],
\end{aligned}
\end{align}where $D(u)$ is an arbitrary and $R(u,\bar{u})=c_1u+\delta\bar{u}+c_0$, $c_1,\delta,c_0\in\mathbb{R}$.
Thus delay PDE \eqref{p1} admits an exact polynomial solution $
u(x,t)=A_1(t)+A_2(t)x,$
where $A_1(t)$ and $A_2(t)$ are:
\begin{align*}
A_1(t)=&\psi(0)\sum\limits_{n=0}^{\floor*{\frac{t}{\tau}}}\delta^n (t-n\tau)^{\alpha n}E_{\alpha,\alpha n+1}^{n+1}(c_1(t-n\tau)^{\alpha})\\
&+\left[\sum\limits_{n=0}^{\floor*{\frac{t}{\tau}}}\delta^{n+1}(t-n\tau)^{\alpha n+\alpha-1}E_{\alpha,\alpha n+\alpha}^{n+1}(c_1(t-n\tau)^{\alpha})\right]\star\left[\psi(t-\tau)g(t-\tau)\right]\\
&+\sum\limits_{n=0}^{\floor*{\frac{t}{\tau}}}c_0\delta^n (t-n\tau)^{\alpha (n+1)}E_{\alpha,\alpha n+\alpha+1}^{n+1}(c_1(t-n\tau)^{\alpha}),\\
A_2(t)=&\chi(0)\sum\limits_{n=0}^{\floor*{\frac{t}{\tau}}}\delta^n (t-n\tau)^{\alpha n}E_{\alpha,\alpha n+1}^{n+1}(c_1(t-n\tau)^{\alpha})\\
&+\left[\sum\limits_{n=0}^{\floor*{\frac{t}{\tau}}}\delta^{n+1} (t-n\tau)^{\alpha n+\alpha-1}E_{\alpha,\alpha n+\alpha}^{n+1}(c_1(t-n\tau)^{\alpha})\right]\star\left[\chi(t-\tau)g(t-\tau)\right],
\end{align*}
where $A_1(t)=\psi(t)$ and $A_2(t)=\chi(t)$, $t\in[-\tau,0]$.
\subsubsection{Two-dimensional trigonometric solution}
For $D(u)=b_2u^{2}+b_1u+b_0$ and $R(u,\bar{u})=a_0b_2u^3+b_1a_0u^2+c_1u+\delta\bar{u}$, $b_i,a_0,c_1,\delta\in\mathbb{R}$, $i=0,1,2$, the non-linear time-fractional heat equation with source term involving delay \eqref{qrd} reduces to
\begin{eqnarray}\label{p2}
\begin{aligned}
&\dfrac{\partial^{\alpha}u}{\partial t^{\alpha}}=\mathcal{H}_2[u,\bar{u}]\equiv (b_2u^{2}+b_1u+b_0)u_{xx}+a_0b_2u^3+b_1a_0u^2+c_1u+\delta\bar{u},\ t>0,\ \alpha\in(0,1],\\
&u(x,t)=\Psi(x,t)=\chi(t)\cos(\sqrt{a_0}x)+\omega(t)\sin(\sqrt{a_0}x),\, t\in[-\tau,0].
\end{aligned}
\end{eqnarray}Eq. \eqref{p2} admits 2-dimensional invariant subspace $\mathcal{W}_2=\text{Span}\left\{\cos{(\sqrt{a_0}x)},\sin{(\sqrt{a_0}x)}\right\}$as
\begin{align*}
&\mathcal{H}_{2}[A_1\cos{(\sqrt{a_0}x)}+A_2\sin{(\sqrt{a_0}x)}, \bar{A}_1\cos{(\sqrt{a_0}x)}+\bar{A}_2\sin{(\sqrt{a_0}x)}]\\
&=\left[(c_1-a_0b_0)A_1+\delta\bar{A}_1\right]\cos{(\sqrt{a_0}x)}+\left[(c_1-a_0b_0)A_2+\delta\bar{A}_2\right]\sin{(\sqrt{a_0}x)}\in \mathcal{W}_2.
\end{align*}
Hence corresponding to $\mathcal{W}_2,$ we find an exact solution of non-linear time-fractional heat equation with linear source term involving time delay \eqref{p2} as
\begin{eqnarray}
\begin{aligned}
u(x,t)=&\left[\chi(0)\sum\limits_{n=0}^{\floor*{\frac{t}{\tau}}}\delta^n (t-n\tau)^{\alpha n}E_{\alpha,\alpha n+1}^{n+1}(\lambda(t-n\tau)^{\alpha})+\sum\limits_{n=0}^{\floor*{\frac{t}{\tau}}}\delta^{n+1}\int\limits^t_0 (r-n\tau)^{\alpha n+\alpha-1}\right.\\
&\times E_{\alpha,\alpha n+\alpha}^{n+1}(\lambda(r-n\tau)^{\alpha})\chi(t-\tau-r)h(t-\tau-r)dr\Bigg]\cos{(\sqrt{a_0}x)}+\Bigg[\omega(0)\\
&\times \sum\limits_{n=0}^{\floor*{\frac{t}{\tau}}}\delta^n (t-n\tau)^{\alpha n}E_{\alpha,\alpha n+1}^{n+1}(\lambda(t-n\tau)^{\alpha})+\sum\limits_{n=0}^{\floor*{\frac{t}{\tau}}}\delta^{n+1}\int\limits^t_0 (r-n\tau)^{\alpha n+\alpha-1}\\
&\times E_{\alpha,\alpha n+\alpha}^{n+1}(\lambda(r-n\tau)^{\alpha})\omega(t-\tau-r)h(t-\tau-r)dr\Bigg]\sin{(\sqrt{a_0}x)},
\end{aligned}
\end{eqnarray}
where $A_1(t)=\chi(t)$ and $A_2(t)=\omega(t)$, $t\in[-\tau,0],$ and $\lambda=c_1-a_0b_0$.
\section{Generalizations}
\subsection{Extension to non-linear time-fractional PDEs involving the linear terms with multiple time delays}
Consider the non-linear time-fractional PDEs involving the linear terms with multiple/ several time delays for $\alpha>0$, having the form
\begin{align}\label{sev}
\begin{aligned}
&\dfrac{\partial^\alpha u}{\partial t^\alpha}=\mathcal{F}[u,\bar{u}_{i}]\equiv\mathcal{N}[u]+\sum\limits^{m}_{i=1}\delta_{i}\bar{u}_i,\ t> 0,\ m\in\mathbb{N},\\
&u(x,t)=\Omega(x,t),\ t\in[-\tau^*,0],\ \tau^*=\text{max}\{\tau_1,\dots,\tau_m\},
\end{aligned}
\end{align}
where $\ {u}=u(x,t),\ \bar{u}_{i}=u(x,t-\tau_i),\ \tau_i>0,\ \delta_i\in\mathbb{R}$. Here $\mathcal{N}[u]=N\left[x,u,u^{(1)},u^{(2)},\dots,u^{(k)}\right]$ is a non-linear differential operator of order $k\ (k\in\mathbb{N})$, where $u^{(r)}=\dfrac{\partial^r u}{\partial x^r},\ r=1,\dots,k,$
and $\tau_i>0$ $(i=1,2,\dots,m, m\in\mathbb{N})$.

The linear space \eqref{lin} is said to be invariant with respect to the non-linear differential operator $\mathcal{F}[u,\bar{u}_{i}]$ if $\mathcal{N}[\mathcal{W}_{n}]\subseteq \mathcal{W}_n$, \textit{i.e.,} $\mathcal{N}[u]\in \mathcal{W}_n$, for all $u\in \mathcal{W}_n$. If $\mathcal{W}_n$ is an invariant under $\mathcal{F}$, then the invariant condition of $\mathcal{F}$ reduces to the following form
\begin{equation}
\mathcal{L}\left(\mathcal{F}[u,\bar{u}_{i}]\right)\Big|_{\mathcal{L}(u)=0}=a_n\dfrac{d^n\mathcal{F}}{dx^n}+a_{n-1}\dfrac{d^{n-1}\mathcal{F}}{dx^{n-1}}+\dots+a_{1}\dfrac{d\mathcal{F}}{dx}+a_0\mathcal{F}\Big|_{\mathcal{L}(u)=0}=0,\ n\in\mathbb{N},
\end{equation}
where $\mathcal{F}[u,\bar{u}_{i}]=\mathcal{N}[u]+\sum\limits^{m}_{i=1}\delta_{i}\bar{u}_i$ and the constants $a_{n},\dots,a_0$ are to be determined. Then there exists $n$ functions $\Theta_j$ $(j=1,2,\dots,n)$ such that
\begin{equation}\label{ex2}
\mathcal{F}\left[\sum\limits^{n}_{j=1}A_j\varphi_j(x), \sum\limits^{n}_{j=1}\bar{A}_j\varphi_j(x)\right]= \sum\limits^{n}_{j=1}\Theta_j(A_1,\dots, A_n)\varphi_j(x)+\sum\limits^{m}_{i=1}\sum\limits^{n}_{j=1}\delta_{i}\bar{A}_j(t)\varphi_j(x).
\end{equation}
\begin{thm}
If the $n-$dimensional linear space \eqref{lin} is invariant under $\mathcal{F}[u,\bar{u}_i]$, then the non-linear time-fractional PDE with several time delays \eqref{sev} possesses generalized separable solutions of the form
\begin{equation}\label{ex1}
u(x,t)=\sum\limits^{n}_{j=1}A_j(t)\varphi_j(x),
\end{equation}
where the coefficients $A_j(t)$ satisfy the following system of fractional delay ODEs
\begin{align*}
&&\dfrac{d^\alpha A_j(t)}{dt^\alpha}=\Theta_j(A_1(t),\dots, A_n(t))+\sum\limits^{m}_{i=1}\delta_iA_j(t-\tau_i),\ j=1,\dots,n.
\end{align*}
\end{thm}
\begin{proof}
Calculating Caputo fractional derivative of order $\alpha$ to \eqref{ex1} with respect to variable $t$, we obtain
\begin{equation}\label{ex4}
\dfrac{\partial^{\alpha}u(x,t)}{\partial t^{\alpha}}=\sum\limits^{n}_{j=1}\left[\dfrac{d^{\alpha}A_j(t)}{dt^\alpha}\right]\varphi_j(x).
\end{equation}
Since the linear space $\mathcal{W}_n$ is invariant under $\mathcal{F}$, in view of \eqref{ex2}, we have
\begin{eqnarray}\label{ex3}
\begin{aligned}
\mathcal{F}[u,\bar{u}_{i}]=&\mathcal{F}\left[\sum\limits^{n}_{j=1}A_j(t)\varphi_j(x), \sum\limits^{n}_{j=1}{A}_j(t-\tau_i)\varphi_j(x)\right]\\
= &\sum\limits^{n}_{j=1}\Theta_j(A_1(t),\dots, A_n(t))\varphi_j(x)+\sum\limits^{m}_{i=1}\sum\limits^{n}_{j=1}\delta_{i}{A}_j(t-\tau_i)\varphi_j(x).
\end{aligned}
\end{eqnarray}
Substituting \eqref{ex4} and \eqref{ex3} in non-linear time-fractional delay PDE \eqref{sev}, we get
\begin{equation}\label{ex5}
\sum\limits^{n}_{j=1}\left[\dfrac{d^\alpha A_j(t)}{dt^\alpha}-\Theta_j(A_1(t),\dots, A_n(t))-\sum\limits^{r}_{i=1}\delta_i{A}_j(t-\tau_i)\right]\varphi_j(x)=0,\ j=1,\dots,n.
\end{equation}
Using the linear independence of function $\varphi_j$'s,\ $j=1,\dots,n$, Eq. \eqref{ex5} yields
$$\dfrac{d^\alpha A_j(t)}{dt^\alpha}=\Theta_j(A_1(t),\dots, A_n(t))+\sum\limits^{r}_{i=1}\delta_i{A}_j(t-\tau_i),\ j=1,\dots,n,\ n\in\mathbb{N}.$$
\end{proof}

\subsubsection{Exact solution for time-fractional heat equation with two time delays}
 Consider the non-linear time-fractional heat equation with source term involving two time delays having the form
\begin{eqnarray}\label{pp2}
\begin{aligned}
&\dfrac{\partial^{\alpha}u}{\partial t^{\alpha}}=\mathcal{F}[u,\bar{u}_i]\equiv (b_2u^{2}+b_1u+b_0)u_{xx}+a_0b_2u^3+b_1a_0u^2+c_1u+\delta_1\bar{u}_1+\delta_2\bar{u}_2,\ t> 0\\
&u(x,t)=\Omega(x,t)=\psi_1(t)\cos{(\sqrt{a_0}x)}+\psi_2(t)\sin{(\sqrt{a_0}x)},\ t\in[-\tau^*,0],
\end{aligned}
\end{eqnarray}
where $\ \alpha\in(0,1],\ \bar{u}_1=u(x,t-\tau_1)$ and $\bar{u}_2=u(x,t-\tau_2),\ \tau^*=\text{max}\{\tau_1,\tau_2\}$.\\
Eq. \eqref{pp2} admits two-dimensional trigonometric subspace $\mathcal{W}_2=\text{Span}\left\{\cos{(\sqrt{a_0}x)},\sin{(\sqrt{a_0}x)}\right\}$.
Thus an exact solution of time-fractional heat equation \eqref{pp2} is of the form
\begin{equation}\label{pp3}
u(x,t)=A_{1}(t)\cos{(\sqrt{a_0}x)}+A_{2}(t)\sin{(\sqrt{a_0}x)},
\end{equation}
where $A_{1}(t)$ and $A_{2}(t)$ are the coefficients which are determined by solving the following linear fractional delay ODEs
\begin{align}
\begin{aligned}\label{p4}
&\dfrac{d^{\alpha}A_{1}}{dt^{\alpha}}=(c_1-a_0b_0)A_{1}(t)+\delta_1A_{1}(t-\tau_1)+\delta_2A_{1}(t-\tau_2),\\
&\dfrac{d^{\alpha}A_{2}}{dt^{\alpha}}=(c_1-a_0b_0)A_{2}(t)+\delta_1A_{2}(t-\tau_1)+\delta_2A_{2}(t-\tau_2).
\end{aligned}
\end{align}
Here $A_{j}(t)=\psi_{j}(t)$, $t\in[-\tau^{\star},0]$, $j=1,2$, where $\tau^{\star}=\text{max}\left\{\tau_{1},\tau_{2}\right\}$.
Applying the Laplace transform on both sides of equations of the system \eqref{p4}, we have
\begin{align*}s^{\alpha}\hat{A}_{j}(s)-s^{\alpha-1}\psi_{j}(0)&=\left(c_1-a_0b_0\right)\hat{A}_{j}(s)+\delta_1\mathrm{L}\left\{A_{j}(t-\tau_1)\right\}+\delta_2\mathrm{L}\left\{A_{j}(t-\tau_2)\right\},\\
s^{\alpha}\hat{A}_{j}(s)-s^{\alpha-1}\psi_{j}(0)&=\left(c_1-a_0b_0\right)\hat{A}_{j}(s)+\delta_1e^{-\tau_1 s}\int\limits_{-\tau_1}^0 e^{-s\mu}\psi_j(\mu)d\mu+\delta_1e^{-\tau_1 s}\hat{A}_{j}(s)\\
&\ \ \ +\delta_2e^{-\tau_2 s}\int\limits_{-\tau_2}^0 e^{-s\mu}\psi_j(\mu)d\mu+\delta_2e^{-\tau_2 s}\hat{A}_{j}(s),
\end{align*}
which on simplification, takes the form
\begin{eqnarray}\label{pp9}
\begin{aligned}
\hat{A}_{j}(s)=&\psi_{j}(0)\left(\dfrac{s^{\alpha-1}}{s^\alpha-\gamma-(\delta_1e^{-\tau_1 s}+\delta_2e^{-\tau_2 s})}\right)+\delta_1\left(\dfrac{e^{-\tau_1 s}}{s^\alpha-\gamma-(\delta_1e^{-\tau_1 s}+\delta_2e^{-\tau_2 s})}\right)\\
&\int\limits_{-\tau_1}^0 e^{-s\mu}\psi_{j}(\mu)d\mu +\delta_2\left(\dfrac{e^{-\tau_2 s}}{s^\alpha-\gamma-(\delta_1e^{-\tau_1 s}+\delta_2e^{-\tau_2 s})}\right)\int\limits_{-\tau_2}^0 e^{-s\mu}\psi_{j}(\mu)d\mu,\ j=1,2,
\end{aligned}
\end{eqnarray}
where $\gamma=c_1-a_0b_{0}$.\\
Using inverse Laplace transform and convolution theorem in \eqref{pp9}, we get
\begin{align*}
A_{j}(t)=&\psi_{j}(0)\mathrm{L}^{-1}\left\{\dfrac{s^{\alpha-1}}{s^\alpha-\gamma-(\delta_1e^{-\tau_1 s}+\delta_2e^{-\tau_2 s})}\right\}\\
&+\delta_1\mathrm{L}^{-1}\left\{\dfrac{1}{s^\alpha-\gamma-(\delta_1e^{-\tau_1 s}+\delta_2e^{-\tau_2 s})}\right\}\star\mathrm{L}^{-1}\left\{e^{-\tau_1 s}\int\limits_{-\tau_1}^0 e^{-s\mu}\psi_{j}(\mu)d\mu\right\}\\
&+\delta_2\mathrm{L}^{-1}\left\{\dfrac{1}{s^\alpha-\gamma-(\delta_1e^{-\tau_1 s}+\delta_2e^{-\tau_2 s})}\right\}\star\mathrm{L}^{-1}\left\{e^{-\tau_2 s}\int\limits_{-\tau_2}^0 e^{-s\mu}\psi_{j}(\mu)d\mu\right\}.
\end{align*}
First consider
\begin{align*}
&\mathrm{L}^{-1}\left\{\dfrac{s^{\alpha-1}}{s^\alpha-\gamma-(\delta_1e^{-\tau_1 s}+\delta_2e^{-\tau_2 s})}\right\}=\ \mathrm{L}^{-1}\left\{\dfrac{s^{\alpha-1}}{\left(1-\dfrac{\delta_1e^{-\tau_1 s}+\delta_2e^{-\tau_2 s}}{s^{\alpha}-\gamma}\right)(s^{\alpha}-\gamma)}\right\}\\
=&\mathrm{L}^{-1}\left\{\dfrac{s^{\alpha-1}}{s^{\alpha}-\gamma}\left(\sum\limits_{n=0}^{\infty}\dfrac{(\delta_1e^{-\tau_1 s}+\delta_2e^{-\tau_2 s})^n}{(s^{\alpha}-\gamma)^n}\right)\right\},\ \Bigg|\dfrac{\delta_1e^{-\tau_1 s}+\delta_2e^{-\tau_2 s}}{s^{\alpha}-\gamma}\Bigg|<1\\
=&\sum\limits_{n=0}^{\infty}\sum\limits^n_{m=0}\delta_1^{n-m}\delta_2^m\binom{n}{m}\mathrm{L}^{-1}\left\{\dfrac{e^{-s((n-m)\tau_1+m\tau_2)}s^{\alpha-1}}{(s^{\alpha}-\gamma)^{n+1}}\right\}\\
=&\sum\limits_{n=0}^{\infty}\sum\limits^n_{m=0}\delta_1^{n-m}\delta_2^m\binom{n}{m}H\left[t-\left((n-m)\tau_1+m\tau_2\right)\right]\mathrm{L}^{-1}\left\{\dfrac{s^{\alpha-1}}{(s^{\alpha}-\gamma)^{n+1}}\right\}\left(t-\left((n-m)\tau_1+m\tau_2\right)\right)\\
=&\sum\limits_{n=0}^{\infty}\sum\limits^n_{m=0}\delta_1^{n-m}\delta_2^m\binom{n}{m}\left(t-\left((n-m)\tau_1+m\tau_2\right)\right)^{\alpha n}H\left[t-\left((n-m)\tau_1+m\tau_2\right)\right]\\
&\times E_{\alpha,\alpha n+1}^{n+1}(\gamma\left(t-\left((n-m)\tau_1+m\tau_2\right)\right)^{\alpha}).
\end{align*}
Similarly
\begin{align*}
&\mathrm{L}^{-1}\left\{\dfrac{1}{s^\alpha-\gamma-(\delta_1e^{-\tau_1 s}+\delta_2e^{-\tau_2 s})}\right\}=\sum\limits_{n=0}^{\infty}\sum\limits^n_{m=0}\delta_1^{n-m}\delta_2^m\binom{n}{m}\mathrm{L}^{-1}\left\{\dfrac{e^{-s((n-m)\tau_1+m\tau_2)}}{(s^{\alpha}-\gamma)^{n+1}}\right\}\\
=&\sum\limits_{n=0}^{\infty}\sum\limits^n_{m=0}\delta_1^{n-m}\delta_2^m\binom{n}{m}H\left[t-\left((n-m)\tau_1+m\tau_2\right)\right]\left(t-\left((n-m)\tau_1+m\tau_2\right)\right)^{\alpha n+\alpha-1}\\
&\times E_{\alpha,\alpha n+\alpha}^{n+1}(\gamma\left(t-\left((n-m)\tau_1+m\tau_2\right)\right)^{\alpha}).
\end{align*}
Finally, we compute $\mathrm{L}^{-1}\left\{e^{-\tau_i s}\int\limits_{-\tau_i}^0 e^{-s\mu}\psi_{j}(\mu)d\mu\right\}$, $i,j=1,2$.\\
Define $g_i(t):[-\tau_i, \infty)\mapsto[0,1]$ by
$$g_i(t)=\left\{
              \begin{array}{ll}
                0, & \hbox{if}\ t\geq0; \\
                1, & \hbox{if}\ t<0.
              \end{array}
            \right.
$$
The function $\psi_{j}(t)$ is extended to $[-\tau_i,\infty)$ by defining $\psi_{j}(t)=\psi_{j}(0)$ for $t\geq0$, then for $i,j=1,2$
\begin{align*}
\mathrm{L}^{-1}\left\{e^{-\tau_{i} s}\int\limits_{-\tau_{i}}^0 e^{-s\mu}\psi_{j}(\mu)d\mu\right\}
=&\mathrm{L}^{-1}\left\{\int\limits_{0}^\infty e^{-s\xi}\psi_{j}(-\tau_{i}+\xi)g_{i}(-\tau_{i}+\xi)d\xi\right\}\\
=&\mathrm{L}^{-1}\left\{\mathrm{L}\left\{\psi_{j}(-\tau_{i}+\xi)g_{i}(-\tau_{i}+\xi)\right\}\right\}=\psi_{j}(t-\tau_{i})g_{i}(t-\tau_{i}).
\end{align*}
Thus for $j=1,2,$
\begin{align*}
A_{j}(t)=&\psi_{j}(0)\sum\limits_{n=0}^{\infty}\sum\limits^n_{m=0}\delta_1^{n-m}\delta_2^m\binom{n}{m}H\left[t-\left((n-m)\tau_1+m\tau_2\right)\right]\left(t-\left((n-m)\tau_1+m\tau_2\right)\right)^{\alpha n}\\
&\times E_{\alpha,\alpha n+1}^{n+1}(\gamma\left(t-\left((n-m)\tau_1+m\tau_2\right)\right)^{\alpha})\\
&+\sum\limits_{n=0}^{\infty}\sum\limits^n_{m=0}\delta_1^{n-m+1}\delta_2^m\binom{n}{m}H\left[t-\left((n-m)\tau_1+m\tau_2\right)\right]\left(t-\left((n-m)\tau_1+m\tau_2\right)\right)^{\alpha n+\alpha-1}\\
&\times E_{\alpha,\alpha n+\alpha}^{n+1}(\gamma\left(t-\left((n-m)\tau_1+m\tau_2\right)\right)^{\alpha})
\star\left[\psi_{j}(t-\tau_1)g_{1}(t-\tau_1)\right]\\
&+\sum\limits_{n=0}^{\infty}\sum\limits^n_{m=0}\delta_1^{n-m}\delta_2^{m+1}\binom{n}{m}H\left[t-\left((n-m)\tau_1+m\tau_2\right)\right]\left(t-\left((n-m)\tau_1+m\tau_2\right)\right)^{\alpha n+\alpha-1}\\
&\times E_{\alpha,\alpha n+\alpha}^{n+1}(\gamma\left(t-\left((n-m)\tau_1+m\tau_2\right)\right)^{\alpha})
\star\left[\psi_{j}(t-\tau_2)g_{2}(t-\tau_2)\right].
\end{align*}
Note that $H\left[t-\left((n-m)\tau_1+m\tau_2\right)\right]=1$ if $t\geq \widetilde{\tau}\equiv(n-m)\tau_1+m\tau_2$. Thus \begin{align*}
t\geq \max\limits_{0\leq m \leq n} \widetilde{\tau}=\max\limits_{0\leq m \leq n}(n-m)\tau_1+m\tau_2=\left\{
\begin{array}{ll}  n\tau_1, & \tau_1>\tau_2,\\
n\tau_2 & \tau_2>\tau_1,
\end{array}
\right.
\end{align*}\textit{i.e.,} $t\geq n \tau^*,$ where $\tau^*=\text{max}\{\tau_1, \tau_2\}.$ This implies $n \leq {\floor*{\frac{t}{\tau^*}}}.$
Hence, we obtain an exact solution of non-linear time-fractional heat equation with source term involving two time delays \eqref{pp2} as
\begin{align*}
u(x,t)=&\left[\psi_1(0)\sum\limits_{n=0}^{{\floor*{\frac{t}{\tau^*}}}}\sum\limits^n_{m=0}\delta_1^{n-m}\delta_2^m\binom{n}{m}\left(t-\widetilde{\tau}\right)^{\alpha n}E_{\alpha,\alpha n+1}^{n+1}(\gamma\left(t-\widetilde{\tau}\right)^{\alpha})+\sum\limits_{n=0}^{{\floor*{\frac{t}{\tau^*}}}}\sum\limits^n_{m=0}\delta_1^{n-m+1}\delta_2^m\right.\\
&\times \binom{n}{m}\int\limits^t_0 \left(r-\widetilde{\tau}\right)^{\alpha n+\alpha-1} E_{\alpha,\alpha n+\alpha}^{n+1}(\gamma\left(r-\widetilde{\tau}\right)^{\alpha})[\psi_1(t-\tau_1-r)g_1(t-\tau_1-r)]dr\\
&+\sum\limits_{n=0}^{{\floor*{\frac{t}{\tau^*}}}}\sum\limits^n_{m=0}\delta_1^{n-m}\delta_2^{m+1}\binom{n}{m}\int\limits^t_0\left(r-\widetilde{\tau}\right)^{\alpha n+\alpha-1} E_{\alpha,\alpha n+\alpha}^{n+1}(\gamma\left(r-\widetilde{\tau}\right)^{\alpha})\psi_{1}(t-\tau_2-r)\\
&\times g_2(t-\tau_2-r)dr\Bigg]\cos{(\sqrt{a_0}x)}+\left[\psi_2(0)\sum\limits_{n=0}^{{\floor*{\frac{t}{\tau^*}}}}\sum\limits^n_{m=0}\delta_1^{n-m}\delta_2^m\binom{n}{m}\left(t-\widetilde{\tau}\right)^{\alpha n}\times\right.\\
&E_{\alpha,\alpha n+1}^{n+1}(\gamma\left(t-\widetilde{\tau}\right)^{\alpha})+\sum\limits_{n=0}^{{\floor*{\frac{t}{\tau^*}}}}\sum\limits^n_{m=0}\delta_1^{n-m+1}\delta_2^m\binom{n}{m}\int\limits^t_0 \left(r-\widetilde{\tau}\right)^{\alpha n+\alpha-1}E_{\alpha,\alpha n+\alpha}^{n+1}(\gamma\left(r-\widetilde{\tau}\right)^{\alpha})\\
&\times\left[\psi_2(t-\tau_1-r)g_1(t-\tau_1-r)\right]dr+\sum\limits_{n=0}^{{\floor*{\frac{t}{\tau^*}}}}\sum\limits^n_{m=0}\delta_1^{n-m}\delta_2^{m+1}\binom{n}{m}\int\limits^t_0\left(r-\widetilde{\tau}\right)^{\alpha n+\alpha-1}\\
&\times E_{\alpha,\alpha n+\alpha}^{n+1}(\gamma\left(r-\widetilde{\tau}\right)^{\alpha})\psi_{2}(t-\tau_2-r)g_2(t-\tau_2-r)dr\Bigg]\sin{(\sqrt{a_0}x)},
\end{align*}
where $\widetilde{\tau}=\left((n-m)\tau_1+m\tau_2\right),\ A_{j}(t)=\psi_{j}(t)$, $t\in[-\tau^{*},0]$, $\tau^{*}=\text{max}\left\{\tau_{1},\tau_{2}\right\}$, $j=1,2$.
\subsection{Another extension to generalized non-linear time-fractional PDEs with time delay}
\textbf{Generalized Form:} Consider the more generalized non-linear time-fractional PDEs with time delay
\begin{align}\label{ge}
\begin{aligned}
&\dfrac{\partial^{\alpha}u}{\partial t^{\alpha}}=\mathcal{G}[u,\bar{u}], \ \alpha>0,\\
&u(x,t)=\Upsilon(x,t),\ t\in[-\tau,0],
\end{aligned}
\end{align}
where $u=u(x,t),\ \bar{u}=u(x,t-\tau),\ \tau>0.$
Here $\mathcal{G}[u,\bar{u}]=\widetilde{\mathcal{G}}\left[x,u,\bar{u},\dfrac{\partial u}{\partial x},\dfrac{\partial \bar{u}}{\partial x}\dots,\dfrac{\partial^k u}{\partial x^k},\dfrac{\partial^k \bar{u}}{\partial x^k}\right],$ is a non-linear differential operator of order $k\ (k\in \mathbb{N})$ with time delay and $\dfrac{\partial^{\alpha}(\cdot)}{\partial t^{\alpha}}$ is a time-fractional derivatives in the Riemann-Liouville/ Caputo sense.\\
\textbf{Note:} Terms involving delay in time-fractional PDE \eqref{ge} need not be linear.\\
The linear space \eqref{lin} is said to be invariant with respect to the non-linear differential operator $\mathcal{G}[u,\bar{u}]$ if $\mathcal{G}[\mathcal{W}_{n},\mathcal{W}_n]\subseteq \mathcal{W}_n$, \textit{i.e.,} $\mathcal{G}[u,\bar{u}]\in \mathcal{W}_n$, for all $u\in \mathcal{W}_n$. If $\mathcal{W}_{n}$ is invariant under $\mathcal{G}[u,\bar{u}]$, then the invariant condition reduces the following form
\begin{equation}
\mathcal{L}\left(\mathcal{G}[u,\bar{u}]\right)\Big|_{\mathcal{L}(u)=0}=a_n\dfrac{d^n\mathcal{G}}{dx^n}+a_{n-1}\dfrac{d^{n-1}\mathcal{G}}{dx^{n-1}}+\dots+a_{1}\dfrac{d\mathcal{G}}{dx}+a_0\mathcal{G}\Big|_{\mathcal{L}(u)=0}=0,\ n\in\mathbb{N},
\end{equation}
where $\mathcal{G}[u,\bar{u}]$ is the given non-linear differential operator, and the constants $a_{n-1},\dots,a_0$ are to be determined. Then there exists $n$ functions $\Psi_j$ $(j=1,2,\dots,n)$ such that
\begin{equation}
\mathcal{G}\left[\sum\limits^{n}_{j=1}A_j\varphi_j(x), \sum\limits^{n}_{j=1}\bar{A}_j\varphi_j(x)\right]= \sum\limits^{n}_{j=1}\Psi_j(A_1,\dots, A_n,\bar{A}_1,\dots,\bar{A}_n)\varphi_j(x).
\end{equation}
\begin{thm}
If the $n-$dimensional linear space \eqref{lin} is invariant under $\mathcal{G}[u,\bar{u}]$, then the generalized non-linear time-fractional PDE with time delay \eqref{ge} possesses generalized separable solutions of the following form
\begin{equation}
u(x,t)=\sum\limits^{n}_{j=1}A_j(t)\varphi_j(x),
\end{equation}
where the coefficients $A_j(t)$ satisfy the following system of fractional delay ODEs
\begin{align*}
&\dfrac{d^\alpha A_j(t)}{dt^\alpha}=\Psi_j(A_1(t),\dots, A_n(t),{A}_1((t-\tau)),\dots,{A}_n((t-\tau))),\ j=1,\dots,n.
\end{align*}
\end{thm}
\begin{proof} Similar to the proof of Theorem 5.1.
\end{proof}
The Invariant subspaces of specific generalized non-linear time-fractional PDEs will be presented elsewhere.
\section{Conclusions}
In the current article, we have presented a detailed study for finding exact solutions of non-linear time-fractional PDEs involving delay. We present how the generalized non-linear time-fractional reaction-diffusion equations admit several invariant subspaces which further yields several analytical solutions. Given time-fractional PDEs with time delay are reduced to system of fractional delay ODEs by using ISM. By solving this system of fractional ODEs we obtain exact solutions of given delay fractional PDEs that can be represented in the form of polynomial, exponential and trigonometric spaces. Furthermore, we employ the ISM to solve non-linear time-fractional PDEs involving a linear term with several time delays. The effectiveness and utility of the ISM have been illustrated by finding exact solutions for non-linear time-fractional heat equation with source term involving two-time delays.
It may be further noted that the invariant subspace method can also be used to investigate exact solutions of more generalized time-fractional delay PDEs with non-linear term having time delay. Also note that the exact solutions of given generalized non-linear time-fractional reaction-diffusion equations with time delay thus obtained have not been reported in the existing literature. The calculated analytical solutions will play vital role in further research. These results demonstrate that the ISM is a very efficient and effective algorithmic tool to find exact solutions for non-linear time-fractional PDEs with time delay.

\end{document}